\documentclass[11pt]{article}
\usepackage{amsmath,amssymb,amsthm,amscd,eucal}
\usepackage[affil-it]{authblk}
\usepackage{latexsym}
\usepackage{graphicx}
\usepackage[usenames]{color} 
\usepackage{enumitem}
\usepackage[pdfencoding=auto, psdextra]{hyperref}
\usepackage{url}

\usepackage{forest}

\usepackage[vcentermath]{youngtab}
\usepackage{young}

\newtheorem{thm}[equation]{Theorem}
\newtheorem{cor}[equation]{Corollary}

\newtheorem{prop}[equation]{Proposition}

\newtheorem*{conj*}{Conjecture}
\theoremstyle{definition}
\newtheorem{remark}[equation]{Remark}
\newtheorem{exam}[equation]{Example}

\numberwithin{equation}{section}

\newcommand{\FF}{\mathbb{F}}  
\newcommand{\ZZ}{\mathbb{Z}}

\newcommand{\A}{\mathsf{A}} 
\newcommand{\M}{\mathsf{M}} 

\newcommand{\Dom}{\operatorname{Dom}} 
\newcommand{\SD}[1]{\mathsf{SD}_{#1}} 
\newcommand{\PD}[1]{\mathsf{PD}_{#1}} 
\newcommand{\T}[1]{\mathsf{T}_{#1}} 
\newcommand{\UT}[1]{\mathsf{UT}_{#1}}   
\newcommand\chara{\mathsf {char}}
\newcommand{\modd}{\mathsf{mod\,}}  
\newcommand{\ev}{\operatorname{ev}_h}
\newcommand{\eval}[2]{\left. #1 \right|_{#2}} 
\newcommand{\ch}{\operatorname{ch}}           
\newcommand{\alltrees}{\mathbb{T}}            
\newcommand{\pS}[2]{\left\lbrace #1 \atop #2\right\rbrace} 

\newcounter{marg}[section]

\usepackage[normalem]{ulem}
\newcommand\rsout{\bgroup\markoverwith{\textcolor{red}{\rule[0.5ex]{2pt}{0.4pt}}}\ULon}

\title{Normally ordered forms of powers of differential operators and their combinatorics}
\author[1]{Emmanuel Briand\thanks{Partially supported by MTM2016-75024-P and FEDER, and Junta de Andalucia under grants P12-FQM-2696 and FQM-333.}}
\author[2]{Samuel A.\ Lopes\thanks{Partially supported by CMUP (UID/MAT/00144/2013), which is funded by FCT (Portugal) with national (MEC) and European structural funds (FEDER), under the partnership agreement PT2020.}}
\newcommand\CoAuthorMark{\footnotemark[1]}
\author[3]{Mercedes Rosas\protect\CoAuthorMark}
\affil[1]{Departamento Matem\'atica Aplicada I, Universidad de Sevilla.}
\affil[2]{CMUP, Faculdade de Ci\^encias da
Universidade do Porto, Rua do Campo Alegre 687,
4169-007 Porto (Portugal).}
\affil[3]{Departamento de \'Algebra, Universidad de Sevilla.}
\date{}

\begin{document}

\maketitle

\begin{abstract}
We investigate the combinatorics of the general formulas for the powers of the operator $h \partial^k$, where $h$ is a central element of a ring and $\partial$ is a differential operator. This generalizes previous work on the powers of operators $h \partial$. New formulas for the generalized Stirling numbers are obtained.
\end{abstract}

\section{Introduction}\label{S:intro}

There exist \textit{universal} polynomials $U_n$ for computing the normally ordered form of $(h\partial)^n$, where $\partial$ is a derivation of a (noncommutative) ring $A$ and $h$ is a central element of $A$. For instance, since, as operators on $A$, we have:
\begin{align*}
(h \partial)^2  & = h \partial(h) \partial + h^2 \partial^2, \\  
(h \partial)^3  & = h (\partial(h))^2 \partial + h^2 \partial^2(h) \partial + 3 h^2 \partial(h) \partial^2 + h^3 \partial^3, 
\end{align*}
the polynomials $U_2$ and $U_3$ are 
\begin{align*}
U_2 & = y_0 y_1 t + y_0^2 t^2,\\
U_3 & = y_0 y_1^2 t + y_0^2 y_2 t  + 3 y_0^2 y_1 t^2 + y_0^3 t^3, 
\end{align*}
so that $(h \partial)^n$ is obtained from $U_n$ by replacing $t$ with $\partial$ and $y_k$ with $\partial^k (h)$.

These polynomials $U_n$ have been studied as early as 1823 by Scherk in his doctoral dissertation \cite[para.\ 8]{Scherk}, taking $\partial=\frac{d}{dx}$, including the calculation of $U_n$ up to $n=5$ (see also the account of Scherk's dissertation in \cite[appx.\ A]{BlasiakFlajolet}). Scherk's results have been revisited or rediscovered in recent times in \cite{Comtet,BergeronReutenauer,MohammadNoori}, as pointed out in the detailed survey \cite{MansourSchork}.  The coefficients occurring in $U_n$ form a family $c^n_\lambda$ of nonnegative integers indexed by partitions $\lambda$ and displaying very interesting combinatorial properties .  To whet the reader's interest, Table~\ref{table:c} below lists the values $c^n_{\lambda}$ for $n\leq 5$.

In this paper, after reviewing in detail formulas and properties of the polynomials $U_n$ and their coefficients $c^n_{\lambda}$, we generalize these results to a broader family of polynomials $U_{n,d}$, providing the normally ordered form of the operator $(h \partial^d)^n$. We also apply the $U_{n, d}$ to the theory of formal differential operator rings.

The motivation for this work comes from \cite{blo13}, where the powers of the derivation $h\frac{d}{dx}$ appear in the formulas for the action of a certain subalgebra $\A_h$ of the Weyl algebra $\A_1$ on its irreducible modules. 
In \cite[Sec.\ 8]{blo13}, some properties of the expression $U_n$ for the normally ordered form of $h\frac{d}{dx}$ were rediscovered.

\paragraph{Organization of the paper.}
Section \ref{S:univpolys} defines formally the polynomials $U_n$ and presents their main properties. Next, in Section \ref{S:combint}, several simple descriptions of the polynomials $U_n$ are provided, each leading naturally to the next one. The first description is an ``umbral formula" for $U_n$, i.e., a formula from which $U_n$ is obtained by applying a linear map $x^{\alpha}\mapsto y_{\alpha}$ that``lowers" exponents to turn them into indices (Section \ref{SS:combint:umbral}). The summands happen to be naturally indexed by simple combinatorial objects (subdiagonal maps). This gives an explicit presentation of the formula (Section \ref{SS:combint:esdm}). In turn, these combinatorial objects encode increasing trees, which gives another presentation and interpretation of the formula  (Section \ref{SS:combint:it}). Increasing trees can be gathered according to their shapes; this provides a more compact way to describe $U_n$ (Section \ref{SS:combint:ut}). While most of these descriptions appear in previous works, we think we have been able to give here a rather clear presentation for them and explain well the natural connections between them. Section \ref{S:combint} also prepares and motivates the generalizations of Section \ref{S:gen}.

After that, we review in Section \ref{S:combspecial} some remarkable specializations of the polynomials $U_n$ and their coefficients: Eulerian polynomials, Stirling numbers of both kinds and generalized Stirling numbers show up. This motivates the study of the coefficients $c^n_\lambda$ of $U_n$ modulo a prime number and we prove one such result in case $n$ is a power of this prime (Theorem~\ref{T:modp}), thus generalizing known results on the modular behavior of Stirling numbers. 

Section \ref{SS:diagrams} provides a simple combinatorial interpretation of the coefficients of the polynomials $U_n$. 
Following that, Section \ref{S:gen} generalizes the descriptions of the polynomials $U_n$ given in Section \ref{S:combint} and the combinatorial interpretations of their coefficients to a family of polynomials $U_{n,d}$ related to the operator $(h\partial^d)^n$. There, generalizations of Comtet's formula for the coefficients of $U_n$ are obtained (Theorem \ref{C:closed:form:d}). Also, as a byproduct, a new formula and a new combinatorial interpretation of the generalized Stirling numbers are derived (Section \ref{SS:g S}). 

Lastly, in Section \ref{S:nord:fdor}  it is shown how the polynomials $U_{n,d}$ provide the normally ordered form of elements in formal differential operator rings $A[z; \partial]$, including as special cases the Weyl algebra $A_1$ and its family of subalgebras $\A_h$ studied in \cite{blo13, blo15ja, blo15tams}. 

\begin{table}[htbp]
\[
\begin{array}{c|*{12}{c}} 
\nu & {\emptyset} & {\Yboxdim{4pt} \yng(1)} & {\Yboxdim{4pt} \yng(2)} & {\Yboxdim{4pt} \yng(1,1)} & {\Yboxdim{4pt} \yng(3)} & {\Yboxdim{4pt} \yng(2,1)} & {\Yboxdim{4pt} \yng(1,1,1)} & {\Yboxdim{4pt} \yng(1,1,1,1)}& {\Yboxdim{4pt} \yng(2,1,1)}& {\Yboxdim{4pt} \yng(2,2)}& {\Yboxdim{4pt} \yng(3,1)}& {\Yboxdim{4pt} \yng(4)}\\[5pt]
 \hline
 c^1_{\nu}&1 &&&&&&&&&&&\\[0.1cm]
 c^2_{\nu}&1 &1  &&&&&&&&&&\\[0.1cm]
 c^3_{\nu}&1 &3  &1 &1   &&&&&&&&\\[0.1cm]
 c^4_{\nu}&1 &6  &4 &7   &1&4  &1   &&&&&\\[0.1cm]
 c^5_{\nu}&1 &10 &10& 25 &5& 30& 15 &1&11&4&7&1
\end{array}
\]
\caption{The coefficients $c^n_{\nu}$ of the polynomials $U_n$ for $n$ up to $5$. Each partition $\nu$ is represented by its Young diagram. This table conceals the signless Stirling numbers of the first kind (as the sums $\sum_{\nu \vdash n-k} c^n_{\nu}$), the Stirling numbers of the second kind (as the coefficients $c^{n}_{(1^{n-k})}$ indexed by one-column shapes) and the Eulerian polynomials (whose coefficients are the sums $\sum_{\ell(\nu)=k} c^n_{\nu}$).}\label{table:c}
\end{table}

\paragraph{Notations.} Throughout the paper, $\ZZ$ and $\ZZ_{\geq 0}$ 
denote the sets of integers and nonnegative integers, respectively. 
For $n\in \ZZ_{\ge 0}$ , define $[n]=\{ 1, \ldots, n\}$, so in particular $[0]=\emptyset$. Given $k \in \ZZ_{\ge 0}$, the notation $(q)_k$ stands for the \emph{falling factorial} $q(q-1)\cdots(q-k+1)$.

An integer partition $\lambda$ is a weakly decreasing sequence $\lambda_1\geq \cdots\geq \lambda_{\ell}>0$ whose terms are called the parts of $\lambda$.
If the sum of the parts of $\lambda$ is $k$, then we write $|\lambda|=k$ or $\lambda \vdash k$.
The number $\ell$ of parts of $\lambda$ is the \emph{length} of $\lambda$, denoted $\ell(\lambda)$.

Our rings are assumed to be associative and unital, but not necessarily commutative. A derivation of a ring $A$ is an additive endomorphism $\partial$ of $A$ satisfying the Leibniz rule 
\begin{equation*}
\partial(ab)=\partial(a)b+a\partial(b), \quad \forall a, b \in A. 
\end{equation*}
In case $A$ is an algebra over a field $\FF$, it is further assumed that $\partial$ is linear. 
If $A=D[x]$ is a polynomial ring over a ring $D$, we denote its derivation $\frac{d}{dx}$ by $\partial_x$, so that $\partial_x$ is zero on $D$ and $\partial_x(x)=1$. For $h\in A$, we denote $\partial^i(h)$ by $h^{[i]}$, reserving the classical notations $h', h''$, etc.\ and $h^{(i)}$ for the special case $\partial=\partial_x$. In particular, $h^{[0]}=h$. Given an integer partition $\lambda=(\lambda_1,\lambda_2,\ldots,\lambda_\ell)$, we define $h^{[\lambda]}=\prod_{i=1}^\ell h^{[\lambda_i]}$. In the same spirit, given commuting variables $y_1$, $y_2$, $y_3, \ldots$, we define $y_{\lambda}=\prod_{i=1}^\ell y_{\lambda_i}$. In case $\lambda=\emptyset$ (the empty partition), then $y_\emptyset=1=h^{[\emptyset]}$.

\paragraph{Acknowledgments.} S.\ Lopes would like to thank G.\ Benkart and M. Ondrus for the discussions motivating the topic of this paper. He would also like to express his gratitude for the hospitality received at the Universidad de Sevilla during his visit at the onset of this collaboration.

\section{The universal polynomials $U_n$ and their coefficients $c_{\lambda}^n$}\label{S:univpolys}

Let $A$ be an arbitrary (noncommutative) ring. For $h\in A$, the map $h\partial$ is not in general a derivation, except if $h$ is central in $A$. \textbf{Thus, we assume throughout that $h$ is central in $A$.} 

\begin{remark}
Note that the derivation $\partial$ stabilizes the center of $A$. 
In particular, $h$ central implies that $\partial^i(h)$ and $\partial^j(h)$ commute, for all $i, j\geq 0$. In fact, all of our results in this paper hold if we replace the hypothesis that $h$ is central with the weaker hypothesis that  $\partial^i(h)$ and $\partial^j(h)$ commute, for all $i, j\geq 0$. 
\end{remark}

We will show next that there is a \textit{universal} polynomial $U_n$ in the variables $y_0, \ldots, y_{n-1}$ and $t$ such that, upon substituting $y_i=h^{[i]}\in A$ and $t=\partial$, we get the operator $\left(h\partial\right)^n$. By this we mean that we write $\left(h\partial\right)^n$ as a sum of terms of the form $a_k\partial^k$ with $a_k$ a monomial in the $h^{[i]}$, for $i\geq 0$, but otherwise independent of $A$ and $\partial$ (hence the term ``universal''). 

To make the above precise we introduce some definitions. Let $R=\ZZ[y_i; i\geq 0]$ be the integral commutative polynomial ring in the variables $\{ y_i\}_{i\geq 0}$. As a free abelian group, $R$ has a basis given by the monomials in the variables $y_i$. Care is needed when doing the substitution $t=\partial$, as $\partial$ does not commute with multiplication (e.g., $h\partial_x$ and $\partial_x h$ are different operators on $\ZZ[x]$). Thus, let $R\langle t \rangle$ be the ring obtained from $R$ by adjoining a new variable $t$ which does not commute with the variables $y_i$; in other words, $R\langle t \rangle$ is the unital ring generated over $\ZZ$ by $\{ y_i\}_{i\geq 0}$ and $t$, subject only to the relations $y_i y_j=y_j y_i$ for all $i, j\geq 0$. Let $\Delta$ be the derivation of $R$ defined on the generators by $\Delta(y_i)=y_{i+1}$ for all $i\geq 0$. Specifically, $\Delta=\sum_{i\geq 0}y_{i+1}\partial_{y_i}$. This map can be extended uniquely to a derivation of $R\langle t \rangle$, still denoted $\Delta$, satisfying $\Delta(t)=0$.

The polynomials $U_n$ are defined recursively, as elements of $R\langle t \rangle$, as follows:
\begin{equation}\label{E:univpolys:recdef}
U_0=1 \quad \text{and} \quad 
\forall n \geq 0,\  U_{n+1} = y_0 (\Delta + \rho_t) U_{n}=y_0 U_n t + y_0 \Delta(U_n),
\end{equation}
where $\rho_t$ denotes the right multiplication by $t$ operator on $R\langle t \rangle$.
It is clear from the definition that $U_n=(y_0 \Delta + y_0 \rho_t)^n (1)$, and that this polynomial depends only on the variables $y_0, \ldots, y_{n-1}$ and $t$. Also, for $n\geq 1$, we can write $U_n=\sum_{i=1}^{n}P_{n, i} t^i$, for some $P_{n, i}\in R$. Moreover, it can be seen that $U_n$ is homogeneous of degree $n$ relative to two different gradings on $R\langle t \rangle$:
\begin{enumerate}[label=\textup{(\roman*)}]
\item the grading in which $y_i$ has degree $1$ and $t$ has degree $0$;
\item the grading in which $y_i$ has degree $i$ and $t$ has degree $1$.
\end{enumerate}

With this we can show that the polynomials $U_n$ determine the normally ordered form of $\left(h\partial\right)^n$, for any ring $A$ and derivation $\partial$, and as such are universal, in an appropriate sense.

\begin{thm}\label{T:univpolys:univprop}
For any $n\geq 0$, there is a unique polynomial $U_n\in \bigoplus_{i\geq 0}Rt^i\subseteq R\langle t \rangle$ such that, for any ring $A$, derivation $\partial$ of $A$ and central element $h$ in $A$,
\begin{equation}\label{E:univpolys:univprop}
\left(h\partial\right)^n 
= \eval{U_n}{y_i=h^{[i]},\, t=\partial}=U_n(h, h^{[1]}, h^{[2]}, \ldots; \partial),
\end{equation}
as endomorphisms of $A$.
\end{thm}
\begin{proof}
Formally, the evaluation on \eqref{E:univpolys:univprop} is the result of applying to $U_n$ the ring homomorphism $\ev:R\langle t \rangle\longrightarrow \operatorname{End}(A)$ sending $y_i$ to $h^{[i]}$, seen as left multiplication by $h^{[i]}$, and $t$ to $\partial$. This is a well-defined map since the elements $h^{[i]}$ commute. Moreover, 
\begin{equation}\label{E:univpolys:skewrel}
\partial\circ\ev(w)=\ev ((\Delta+ \rho_t)(w))
\end{equation}
holds in $\operatorname{End}(A)$, for all $w\in R\langle t \rangle$. (This could be seen as a version of the chain rule, adapted for our context.)
Indeed, using the additive and multiplicative properties of $\ev$ and $\Delta$ we see that it suffices to check this identity on a generating set of $R\langle t \rangle$. In case $w=t$, the identity is trivial because $\Delta(t)=0$; in case $w=y_i$ we obtain the equality $\partial\circ h^{[i]}=h^{[i+1]}+h^{[i]}\circ\partial$, where, as before, $h^{[j]}$ is seen as left multiplication by $h^{[j]}$. Thus, the latter equality follows from the fact that $\partial$ is a derivation of $A$.

It is now an easy matter to prove \eqref{E:univpolys:univprop} by induction on $n$, using \eqref{E:univpolys:skewrel}. The base case is clear, as $U_0=1$, and for the inductive step we have 
\begin{align*}
(h \partial)^{n+1}&=h\circ\partial\circ \ev(U_n)\\
&=h\circ\ev((\Delta+\rho_t)(U_n)) \\
&=\ev(y_0\Delta(U_n)+y_0 U_n t)\\
&=\ev(U_{n+1}).
\end{align*} 

To show uniqueness, suppose $V_n\in \bigoplus_{i\geq 0}Rt^i$ satisfies \eqref{E:univpolys:univprop}. Take $A=R$, $h=y_0$ and $\partial=\Delta$, so that $\partial(y_i)=y_{i+1}$ for all $i\geq0$. Let $T=U_n-V_n$. Then we can write $T=\sum_{i=0}^{m}P_{i} t^i$, for some $m\geq 0$ and $P_{i}\in R$. Thus, as in $R$ we have $h^{[i]}=\Delta^i(y_0)=y_i$, we obtain $0=\eval{T}{y_i=h^{[i]},\, t=\Delta}=\sum_{i=0}^{m}P_{i} \Delta^i$. Applying this operator to $y_k$, where $k$ is large enough so that no $y_j$ with $j\geq k$ occurs in any of the $P_i$, we conclude that $\sum_{i=0}^{m}P_{i} y_{k+i}=0$, from which follows that $P_i=0$ for all $i$ and $V_n=U_n$.
\end{proof}

Our next goal is to set a recurrence relation for the coefficients of $U_n$.

\begin{prop}[{\cite[par.\ 8.I.1]{Scherk}, \cite[Thm.\ 1]{MohammadNoori}, \cite[Sec.\ 8]{blo13}}]\label{P:recrel}
Assume $n\geq 1$. There exist positive integers $c^n_\lambda$, where $\lambda$ runs through the set of integer partitions of size $0\leq |\lambda|< n$, such that
\begin{equation}\label{E:un:canonical}
U_n=\sum_{k=1}^n\sum_{\lambda\vdash n-k} c^n_\lambda y_0^{n-\ell(\lambda)}y_{\lambda}t^{k}. 
\end{equation}
Additionally, the coefficients $c^n_\lambda$ satisfy the recurrence relation
\begin{equation}\label{E:prop:recrel}
c^1_\emptyset=1, \quad c^{n+1}_\lambda=c^{n}_\lambda + \sum_{i=1}^n (\beta_{i-1}+1) c^{n}_{\lambda_i},
\end{equation}
where:
\begin{itemize}
\item $\beta_0=n-\ell(\lambda)$ and, for $j\geq 1$, $\beta_j$ is the multiplicity of $j$ in $\lambda$;
\item $\lambda_i$ is obtained from $\lambda$ by subtracting $1$ from a part of $\lambda$ of size $i$, provided that $\beta_i>0$;
\item $c^{n}_{\lambda_i}=0$ if  $\beta_i=0$;
\item $c^{n}_{\lambda}=0$ if $\lambda\vdash m$ with $m\geq n$.
\end{itemize}
\end{prop}

\begin{proof}
As we have observed before, $U_n$ is a sum of integer multiples of monomials of the form $y_0^{a_0}\cdots y_{n-1}^{a_{n-1}}t^k$, with $1\leq k\leq n$. Given such a monomial occurring in $U_n$ with nonzero coefficient, we need to verify that:
\begin{enumerate}[label=\textup{(\alph*)}]
\item $0\leq 1a_1+ 2a_2+\cdots + (n-1)a_{n-1}\leq n-1$, so that $y_1^{a_1}\cdots y_{n-1}^{a_{n-1}}=y_\lambda$, where $\lambda$ is the partition with $a_i$ parts of size $i$;
\item $a_0=n-\ell(\lambda)=n-(a_1+\cdots +a_{n-1})$;
\item $k=n-|\lambda|=n-(1a_1+\cdots +(n-1)a_{n-1})$.
\end{enumerate}
These can be easily established by induction on $n$ using \eqref{E:univpolys:recdef} (actually, they are equivalent to the two homogeneous properties of $U_n$ already mentioned) and so can the recurrence relation \eqref{E:prop:recrel}. Finally, the positivity of the coefficients $c^n_\lambda$ can also be established by induction on $n$, using \eqref{E:prop:recrel}.
\end{proof}

\begin{exam}
For $\nu=(2, 1)$, we have $$c^5_{\nu}=c^4_{\nu}+3c^4_{\nu_1}+2c^4_{\nu_2}=c^4_{2, 1}+3c^4_{2}+2c^4_{1, 1}=4+3\times 4+2\times 7=30,$$
and for $\nu=(2, 1, 1)$, we have $$c^5_{\nu}=c^4_{\nu}+2c^4_{\nu_1}+3c^4_{\nu_2}=2c^4_{2, 1}+3c^4_{1, 1, 1}=2\times 4+3\times 1=11.$$ 
\end{exam}

\begin{remark}
Proposition \ref{P:recrel} tells us the size of the polynomial $U_n$: its number of monomials is the number of integer partitions of size at most $n-1$ (up to an offset, this is sequence A000070 in \cite{oeis}).
\end{remark}


\section{Combinatorial interpretations of the polynomials $U_n$}\label{S:combint}

In this section we introduce several descriptions of the polynomials $U_n$, each leading naturally to the next one. The first description is an ``umbral formula" for $U_n$, i.e., a formula from which $U_n$ is obtained by applying a linear map $x^{\alpha}\mapsto y_{\alpha}$ that ``lowers" exponents to turn them into indices (Section \ref{SS:combint:umbral}). The summands happen to be naturally indexed by simple combinatorial objects (strictly subdiagonal maps). This gives an explicit presentation of the formula for $U_n$ (Section \ref{SS:combint:esdm}). In turn, these combinatorial objects encode increasing trees, giving another presentation and interpretation of the formula  (Section \ref{SS:combint:it}). Increasing trees can be gathered according to their shapes; this provides a more compact way of describing $U_n$ (Section \ref{SS:combint:ut}).

Most of the results and ideas  in this section are not new: they already appear in the paper by Mohammad--Noori \cite{MohammadNoori} or the (seemingly unrelated) work of Hivert, Novelli and Thibon \cite{HNT} on the solutions of certain differential equations. We have tried to present them as clearly and explicitly as possible. This serves for preparing the generalizations  of Section \ref{S:gen}.

\subsection{Umbral formula}\label{SS:combint:umbral}

We continue to assume that $A$ is an arbitrary ring and $\partial$ is a derivation of $A$. Consider the tensor product (over $\ZZ$):
\[
A^{\otimes n} = A \otimes A \otimes \cdots \otimes A \quad \textrm{($n$ copies of $A$),}
\]
where we index the factors in the tensor product from right to left and from $0$ to $n-1$:
\[
n-1, n-2, \ldots, 1, 0.
\]
We denote by $m_n: A^{\otimes n} \longrightarrow A$ the $n$-ary multiplication map and by $\partial_{n,i}$ the map $A^{\otimes n} \longrightarrow A^{\otimes n}$ that applies $\partial$ on factor $i$ and the identity on the others. If $i\leq n-1$, then $\partial_{n+1, i}=1_A\otimes \partial_{n,i}$ so, to lighten the notation, we henceforth use $\partial_i$ to denote any of the maps $\partial_{n,i}$, as long as $n\geq i+1$.

For $n=2$, the Leibniz identity can be stated as
\[
\partial \circ m_2 = m_2 \circ (\partial_1+\partial_0).
\]
By an easy induction, we get that, for any $n \geq 1$,
\begin{equation}\label{E:umbral:dm}
\partial \circ m_n = m_n \circ \sum_{i=0}^{n-1} \partial_i. 
\end{equation}
Using this identity, it is derived, again by induction, that for any $n \geq 0$ and any $h_0$, $h_1$, \ldots, $h_n \in A$,
\[
h_n \partial h_{n-1} \partial \cdots h_1 \partial (h_0) 
= m_{n+1} \circ
\left(
  \prod_{i=0}^{n-1}   \left(\partial_{i}+\cdots+\partial_{1} + \partial_0\right)\right)
  \left(h_{n}  \otimes \cdots \otimes h_1 \otimes h_0\right).
\]
In particular, for any $a,h \in A$ and any $n \ge 0$,
\[
(h \partial)^n(a) = 
m_{n+1} \circ
\left(
\prod_{i=0}^{n-1} \left(  \partial_{i}+\cdots+\partial_{1} + \partial_0\right) \right)
\left(h  \otimes \cdots \otimes h \otimes h \otimes a \right).
\]
This yields the following recipe for calculating $U_n$.
\begin{thm}\label{T:umbral}
For any $n \ge 0$, the polynomial $U_n$ is obtained by applying 
to the product
\begin{equation}\label{E:product}
\prod_{i=0}^{n-1} \left( x_{i}+ \cdots + x_2 + x_1 + x_0\right)
\end{equation}
in the commutative polynomial ring $\ZZ[x_0, x_1,\ldots,x_n]$ the $\ZZ$-linear map from $\ZZ[x_0, x_1,\ldots,x_n]$ to $R\langle t \rangle$ defined on a basis of $\ZZ[x_0, x_1,\ldots,x_n]$ by
\[
x_n^{\alpha_n} x_{n-1}^{\alpha_{n-1}} \cdots x_1^{\alpha_1}x_0^k \longmapsto y_{\alpha_n} y_{\alpha_{n-1}} \cdots y_{\alpha_1}t^k.
\]
\end{thm}

\begin{remark}
The product in~\eqref{E:product} contains no occurrence of $x_n$. This is deliberate so that $y_0$ is always a factor of $U_n$, for $n>0$. Notice in particular that the term $x_0^n$ from~\eqref{E:product} is mapped to $y_0^n t^n\in R\langle t \rangle$, which is always a term of $U_n$.
\end{remark}


\subsection{Enumeration of  subdiagonal maps}\label{SS:combint:esdm}

The expansion of the product~\eqref{E:product} in $\ZZ[x_0, x_1,\ldots,x_n]$
is obtained by choosing for the $i$-th factor $\left( x_{i-1}+ \cdots + x_2 + x_1 + x_0\right)$ a summand $x_{f(i)}$, and summing over all possibilities. This expansion is thus
\[
\sum_{f \in \SD{n}} x_{f(n)} x_{f(n-1)} \cdots x_{f(2)}  x_{f(1)}
\]
where $\SD{n}$ is the set of all maps $f:  \{1,\ldots,n\} \longrightarrow \{0,1,\ldots,n\}$ that are  \textit{subdiagonal}, i.e., that satisfy $f(i) < i$ for all $1\leq i\leq n$. In particular, $f([n])\subseteq \{0,1,\ldots,n-1\}$ but it will be convenient to take the codomain of elements in $\SD{n}$ to be $\{0,1,\ldots,n\}$.

Gathering the variables, we get that this expansion is:
\[
\sum_{f \in \SD{n}} 
\prod_{i=0}^{n-1} x_i^{\#f^{-1}(\{i\})}.
\]
In conjunction with Theorem~\ref{T:umbral}, the above gives another description of $U_n$.
\begin{thm}\label{T:SD}
For any $n \ge 0$, 
\[
U_n  = \sum_{f \in \SD{n}}   \left(\prod_{i=1}^{n} y_{\#f^{-1}(\{i\})}\right) t^{\#f^{-1}(\{0\})}. 
\]
\end{thm}
It will be interesting, when studying the coefficients $c^n_{\lambda}$, to consider rather than the functions $f \in \SD{n}$, the corresponding partial maps from $[n]$ to $[n]$. A \emph{partial map} $g$ from a set $A$ to a set $B$ is a map from some subset, $\Dom(g)$, of $A$ to $B$. Let $\PD{n}$ be the set of all subdiagonal partial maps from $[n]$ to $[n]$, i.e., all maps $g$ from some subset $\Dom(g)$ of $[n]$ to $[n]$, satisfying $g(i) < i$, for all $i \in \Dom(g)$.

\begin{thm}\label{T:PD}
For any $n \ge 0$,  
\begin{align*}
U_n &= \sum_{g \in \PD{n}} 
\left(\prod_{i=1}^{n} y_{\#g^{-1}(\{i\})}\right) t^{n-\#\Dom(g)}. 
\end{align*}
\end{thm}


\subsection{Expansion in increasing trees}\label{SS:combint:it}

Following Cayley \cite{Cayley},  Hivert, Novelli and Thibon \cite{HNT}, Mohammad--Noori \cite{MohammadNoori}, and Blasiak and Flajolet \cite{BlasiakFlajolet}, we interpret the calculations of iterated derivatives in term of increasing trees. This is done by interpreting the map $f\in \SD{n}$ as the map that associates to each node $j\in [n]$ its father in a tree rooted at $0$. This gives a bijection from $\SD{n}$ to the set $\T{n}$ of all increasing rooted trees with vertex set $\{0,1,\ldots,n\}$. Then, for each $i$, $\#f^{-1}(\{i\})$ is the number of children of node $i$. From this and Theorem \ref{T:SD} follows the following formula.
\begin{thm}[\cite{MohammadNoori,BergeronReutenauer}]
\label{T:trees}
For any $n \ge 0$,
\begin{equation}\label{E:trees}
U_n = \sum_{T \in \T{n}} \left(\prod_{i=1}^n y_{\ch(i;T)}\right) t^{\ch(0;T)},
\end{equation}
where $\ch(i;T)$ stands for the number of children (outdegree) of node $i$ in the rooted tree $T$.
\end{thm}


\subsection{Expansion in unlabeled trees}\label{SS:combint:ut}

In \eqref{E:trees}, the monomial  $\left(\prod_{i=1}^n y_{\ch(i;T)}\right)t^{\ch(0;T)}$ depends only on the shape of the tree $T$, i.e., the unlabeled rooted tree obtained from $T$ by forgetting the labels. This can be used to turn Formula \eqref{E:trees} into a formula with considerably fewer summands. Let $\UT{n}$ be the set of unlabeled rooted trees with $n+1$ vertices and, for any $T\in \UT{n}$, let $\alpha(T)$ be the number of increasing trees with vertex set $\{0,1,\ldots,n\}$ and shape $T$.

\begin{thm}
For any $n \ge 0$,
\begin{equation}\label{E:unlabeled:trees}
U_n = \sum_{T \in \UT{n}} \alpha(T)  \left( \prod_{v \neq 0_T} y_{\ch(v;T)}\right)t^{\ch(0_T;T)},
\end{equation}
where  $0_T$ stands for the root of $T$ and the product is carried over all vertices $v$ of $T$ different from the root.
\end{thm}

The coefficients $\alpha(T)$ appear in \cite{BroadhurstKreimer, HNT} as the \emph{Connes-Moscovici} coefficients. Given an unlabeled rooted tree $T$ with $n+1$ vertices,  let $b_1$, $b_2$, \ldots, $b_k$ be its main branches, i.e., the connected components of $T-\{0_T\}$, whose roots are the children of the root of $T$. Some of the $b_i$ may be equal. Let $\alltrees=\bigcup_{j\geq 0}\UT{j}$ be the set of all unlabeled rooted trees. For any $s \in \alltrees$, let $m_T(s)$ be the multiplicity of $s$ in the sequence $(b_1, \ldots, b_k)$. Then $m_T$ is a function from $\alltrees$ to $\ZZ_{\geq 0}$. We have thus
\begin{equation}\label{alpha}
\alpha(T) = \beta(m_T) \prod_{s \in \alltrees} \alpha(s)^{m_T(s)},
\end{equation}
where $\beta(m_T)$ is the number of set partitions of $[n]$, with blocks labeled with elements of $\alltrees$, such that for each $s \in \alltrees$, there are exactly $m_T(s)$ blocks labeled by $s$, all of size $\#s$, the number of vertices of $s$. 
This number is given by
\begin{equation}\label{beta}
\beta(m_T)=\frac{n!}{\prod_{s \in \alltrees} m_T(s)! (\#s)!^{m_T(s)}}. 
\end{equation}


\section{Special values of the coefficients $c_{\lambda}^n$}\label{S:combspecial}

In this section, we derive some interesting specializations and properties of the polynomials $U_n$ and their coefficients $c_{\lambda}^n$. Some of the results in this section appear in \cite{Comtet, MohammadNoori, BergeronReutenauer}, others are new: in particular, the observation that the generalized Stirling numbers $\pS{n}{k}_{q,1}$ appear in specializations of the polynomials $U_n$ (Section \ref{S:combspecial:gs2}), and the vanishing modulo $p$ of most of the coefficients of $U_n$ when $p$ is prime and $n$ is a power of $p$ (Section \ref{SS:mod p}).

\subsection{Factorials and Stirling numbers of the first kind}\label{S:combspecial:fs1}

If we set $y_i=1$ for all $i\geq 0$ in $U_n$ then, by Theorem~\ref{T:umbral}, this corresponds to taking $x_j=1$ for all $j\geq 1$ and $x_0=t$ in \eqref{E:product}, so we get
\begin{equation}\label{E:combspecial:fs1:rf}
U_n(1, 1,1 \ldots ; t)=\prod_{i=0}^{n-1} (t+i)=t(t+1)\cdots (t+n-1),
\end{equation}
a rising factorial. Since the latter expression coincides with the generating function $\sum_{k=1}^n c(n, k)t^k$ for the (signless) Stirling numbers of the first kind, $c(n, k)$, which counts the number of permutations in $S_n$ with exactly $k$ cycles, we deduce from \eqref{E:un:canonical} that 
\begin{equation}
\sum_{\lambda\vdash n-k} c^n_\lambda=c(n, k).
\end{equation}
This result appears in \cite[\S 5]{Comtet}, \cite[Prop.\ 9 (iii)]{MohammadNoori} and \cite[p.\ 274]{BergeronReutenauer}.

In particular, for $n\geq 1$,
\begin{align*}
\sum_{k=1}^n\sum_{\lambda\vdash n-k} c^n_\lambda =\sum_{k=1}^n c(n, k)=U_n(1, \ldots, 1, 1)=n!
\end{align*}
and 
\begin{align*}
\sum_{\lambda\vdash n-1} c^n_\lambda =c(n, 1)=(n-1)!.
\end{align*}

We remark that all of the above relations have straightforward bijective proofs. For example, by \eqref{E:trees}, we have that $\sum_{\lambda\vdash n-k} c^n_\lambda$ is the number of increasing rooted trees $T$ on the vertex set $\{0,1,\ldots,n\}$ for which the root has exactly $k$ children. By \cite[Prop.\ 1.5.5]{rS12}, this number is equal to $c(n, k)$, because there is a bijective correspondence between $\T{n}$ and the symmetric group $S_n$ under which trees with a root having $k$ children correspond to permutations in $S_n$ with $k$ left-to-right maxima, which in turn, under the \textit{fundamental bijection} (\cite[Prop.\ 1.3.1]{rS12}), correspond to permutations in $S_n$ with $k$ cycles.

\begin{remark}
Specializing $y_i$ at $q^i$ instead of $1$ does not provide any new identity, due to the homogeneity properties of the polynomials $U_n$: 
\[
U_n(1,q,q^2,\ldots; t) = q^n U_n\left(1,1,1,\ldots,\frac{1}{q} t\right).
\]
\end{remark}

\subsection{Generalized Stirling numbers, Stirling numbers of the second kind and Bell numbers}\label{S:combspecial:gs2}

The \emph{partition Stirling numbers}, or \emph{Stirling numbers of the second kind}, $\pS{n}{k}$, count the set partitions in $k$ blocks of a set with $n$ elements. It is well-known that the partition Stirling numbers are precisely the coefficients of the normally ordered form of $(x \partial_x)^n$, for $A=\ZZ[x]$:
\[
(x \partial_x)^n = \sum_k \pS{n}{k} x^k \partial_x^k.
\]
This identity leads to a natural generalization of the partition Stirling numbers, simply called \emph{generalized Stirling numbers} and denoted $\pS{n}{k}_{q,d}$: these are the coefficients in the normally ordered form of $(x^q \partial_x^d)^n$ in $\ZZ[x]$. More precisely:
\begin{equation}\label{E:def:Stirling}
\begin{aligned}
 (x^q \partial_x^d)^n &= x^{(q-d) n} \ \sum_k \pS{n}{k}_{q,d} x^k \partial_x^k  &\textrm{ if $q \ge d$,}\\
(x^q \partial_x^d)^n &=\phantom{x^{(q-d)}} \left(\sum_k \pS{n}{k}_{q,d} x^k \partial_x^k\right) \partial_x^{(d-q)n}  & \textrm{ if $q \le d$.}
\end{aligned}
\end{equation}
The generalized Stirling numbers 
have been studied extensively (see specially \cite[4.7.1]{Blasiak:thesis}, \cite[Note 18]{BlasiakFlajolet} and the references in the survey \cite{MansourSchork}), and in particular the question of finding a simple combinatorial interpretation for them has been raised. This problem was solved (even for more general monomials in $x$ and $\partial_x$) in \cite{Varvak}, in terms of rook placements on chessboards that are Young diagrams and, in \cite{EngbersGalvinHilyard}, in terms of constrained partitions of vertex sets of graphs.

Observe that, for $d=1$, the right-hand side of \eqref{E:def:Stirling}  is equal to the specialization $U_n(h,h',h'',\ldots;\partial_x)$, with $h=x^q$. Evaluate at $x=1$, and use that, for all $k$, $h^{(k)}(1)=(q)_k=q(q-1)\cdots(q-k+1)$, the falling factorial, to get
\begin{equation}\label{specialization at q_i}
U_n\left( (q)_0, (q)_1, (q)_2 , \ldots; \partial_x \right) =  \sum_k \pS{n}{k}_{q,1} \partial_x^k.
\end{equation}
Therefore, the specialization of $U_n$ at $y_k=(q)_k$, for $q$ a nonnegative integer, is the ordinary generating series for the generalized Stirling numbers $\pS{n}{k}_{q,1}$ and it follows that 
\begin{equation}\label{E:stirling q 1}
\pS{n}{k}_{q,1} = \sum_{\lambda\vdash n-k} c_{\lambda}^{n} \prod_i (q)_{\lambda_i}.
\end{equation}

The generalized Stirling numbers $\pS{n}{k}_{q,d}$, for arbitrary $d$, will be obtained in the same fashion from the polynomials $U_{n,d}$
in Section \ref{S:gen}.

Further specialization of \eqref{specialization at q_i} at $q=1$ gives 
\[
\eval{U_n}{y_0=y_1=1,\, y_i=0\ \forall i>1}=\sum_{k=1}^n \pS{n}{k}t^k,
\]
which is the generating function for the partition Stirling numbers (\emph{Touchard polynomials}).
In particular, $\displaystyle c^n_{1^{n-k}}=\pS{n}{k}$, a result that appears in \cite[\S 5]{Comtet}, \cite[Prop.\ 9 (ii)]{MohammadNoori} and \cite[p.\ 275]{BergeronReutenauer}. Taking $t=1$ we obtain the Bell number 
\[
B_n=\sum_{k=1}^n \pS{n}{k}=\eval{U_n}{y_0=y_1=t=1,\, y_i=0\ \forall i>1},
\]
which counts the number of set partitions of $[n]$.

For example, using the recurrence relation \eqref{E:prop:recrel}, we obtain the well-known relation for the Stirling numbers:
\begin{align*}
\pS{n}{k} =c^n_{1^{n-k}}=c^{n-1}_{1^{n-k}}+kc^{n-1}_{1^{n-k-1}}=\left\{ {n-1\atop k-1}\right\}+k\left\{ {n-1\atop k}\right\}.
\end{align*}

\subsection{Eulerian polynomials}\label{S:combspecial:hnte}

The coefficient of $q^k$ in $U_n(q, 1, 1, 1,\ldots; 1 )$
is the number of increasing trees $T\in \T{n}$  with exactly $k$ leaves. By \cite[Prop.\ 1.5.5]{rS12}, this is precisely the Eulerian number $A(n, k)$, equal to the number of permutations in $S_n$ with exactly $k-1$ descents. Thence, this specialization of $U_n$ is the Eulerian polynomial $A_n(q)$:
\[
U_n(q, 1, 1, \ldots; 1)= \sum_{\pi\in S_n} q^{1+\text{no.\ of descents of $\pi$}}=A_n(q).
\]
The same argument also shows that the coefficient of $q^k$ in $U_n(1, q, q, \ldots; 1)$ is the Eulerian number $A(n, n-k)$. Since the Eulerian polynomials are palindromic (i.e., $q^{n+1} A_n(1/q) = A_n(q)$), these properties, in conjunction, show that we have 
\[
A_n(q) = U_n(q, 1, 1, \ldots; 1)=qU_n(1, q, q, \ldots; 1).
\] 
Thus, for all $k$, 
\begin{equation}\label{E:special:A(n, k)}
\sum_{\lambda: \ell(\lambda)=k} c_{\lambda}^n =A(n, k+1)=\mbox{no.\ of $\pi\in S_n$ with exactly $k$ descents.}
\end{equation}
The above formula is implicit in \cite[Prop.\ 9 (iv)]{MohammadNoori}, as well as \cite[p.\ 275]{BergeronReutenauer}.

\subsection{The solution of the differential equation $x'(u) = y(x(u))$}

Given a formal series 
$
y(u)=\sum_{i=0}^{\infty} y_i \frac{u^i}{i!},
$
there is a unique formal series
$
x(u)=\sum_{n=1}^{\infty} x_n \frac{u^n}{n!}.
$
such that
$
\frac{dx}{du} = y(x(u)).
$
Hivert, Novelli and Thibon show in \cite[Formula (43)]{HNT} that its coefficients are given by 
\[
x_n = \sum_{T \in \T{n-1}} \prod_{i=0}^{n-1} y_{\ch(i;T)},
\]
which is the image of $U_n$ under the
left $R$-module map $\bigoplus_{k\geq 0}Rt^k\longrightarrow R$
that sends $t^k$ to $y_k$, for all $k \ge 0$,

\subsection{The coefficients $c^n_{\lambda}$ modulo a prime $p$, for $n=p^m$}\label{SS:mod p}

It is well-known that most Stirling numbers of both kinds $c(p,k)$ and $\pS{p}{k}$ vanish modulo $p$ when $p$ is prime. We will see that this property is shared by the coefficients $c^p_\lambda$ and, more generally, by $c^n_\lambda$, for $n=p^m$.

\begin{thm}\label{T:modp}
For any prime $p$, prime power $n=p^m$ and partition $\lambda$ with $|\lambda|\neq n-1$ and $|\lambda|$ not a multiple of $p$, we have $c^n_\lambda \equiv 0\ (\modd p)$. In particular, if $|\lambda|\neq 0, p-1$, then $c^p_\lambda \equiv 0\ (\modd p)$.
\end{thm}

Recall from Sections \ref{S:combspecial:fs1} and \ref{S:combspecial:gs2} that
\[
c(n,k)=\sum_{\lambda \vdash n-k} c_{\lambda}^n, \qquad  \pS{n}{k}=c_{1^{n-k}}^n 
\quad \text{and}\quad
\pS{n}{k}_{q,1} = \sum_{\lambda\vdash n-k} c_{\lambda}^{n} \prod_i (q)_{\lambda_i}.
\]
Thus, Theorem \ref{T:modp} says in particular that for any prime $p$, $n=p^m$ and $1<k<n$ not a multiple of $p$, the Stirling numbers of both kinds 
 $c(n,k)$ and $\pS{n}{k}$ as well as the generalized Stirling numbers $\pS{n}{k}_{q,1}$ are multiples of $p$.

\begin{proof}
For each $n>0$, we will build an action of the cyclic group $\ZZ_n$ on $\T{n}$. This action will preserve the outdegree (number of children) of vertex $0$, and will preserve as well the multiset of outdegrees of all other vertices. Therefore, it will be possible to write Formula \eqref{E:trees} as 
\[
U_n = \sum_{\omega \in \T{n}/\ZZ_n}  \# \omega \cdot   \left(\prod_{i=1}^n y_{\ch(i;T_{\omega})}\right) t^{\ch(0;T_{\omega})}.
\]
The sum is carried over all orbits $\omega$ for this action and $T_{\omega}$ is an arbitrarily chosen tree in $\omega$. 

Fix a prime $p$ and let $n=p^m$. We can assume that $m>0$. The cardinality of any orbit $\omega$ will also be a power or $p$, hence either $1$ or a multiple of $p$. We will show that if $T\in\T{n}$ is fixed under the action, then either $\ch(0;T)=1$ or $\ch(0;T)$ is a multiple of $p$, from which the theorem will follow.

The action of $\ZZ_n$ that we will define is based on the bijection between $\T{n}$ and the set $\Sigma_n$ of all sets 
\[
\{ (t_1, A_1), (t_2, A_2), \ldots, (t_k,A_k)\}
\]
where the $A_i$ are the blocks of a set partition of $[n]$ and, for each $i$, $t_i$ is an increasing tree whose number of vertices is $\# A_i$. This bijection is defined as follows: given $T\in \T{n}$, delete its root $0$. An increasing forest with vertex set $[n]$ is obtained. Let $t'_1$, $t'_2$, \ldots, $t'_k$ be its components (each is an increasing tree). Define $A_i$ as the set of vertices of $t'_i$. Enumerate the elements of $A_i$ in increasing order: $A_i=\{v_1< v_2 < \cdots < v_r\}$; and define $t_i$ as the tree obtained from $t'_i$ by replacing, for each $j$, the number  $v_j$ with its index $j$.

The natural action of the cyclic group $\ZZ_n$ on $[n]$ induces naturally an action on the set of all set partitions of $[n]$, and this in turn induces an action on $\Sigma_n$. The bijection $\Sigma_n \cong \T{n}$ is then used to transport this action to an action on $\T{n}$.

Assume $T\in\T{n}$ is fixed under this action of $\ZZ_n$ and let $\{ A_1, \ldots, A_k \}$ be the corresponding set partition of $[n]$. In particular, $\ch(0;T)=k$ and it remains to show that either $k=1$ or $p$ divides $k$. Since $T$ is fixed, $\ZZ_n$ permutes the blocks of the partition and we consider this induced action next. If a certain block $A_i$ is fixed, then necessarily $A_i=[n]$ and thus $k=1$. Otherwise, the size of the orbit of each block is divisible by $p$, thence so is $k$.
\end{proof}

\begin{remark}
It is also possible (and elementary) to prove the $n=p$ case of Theorem \ref{T:modp} using Formula \eqref{E:unlabeled:trees} and showing that the Connes--Moscovici coefficients $\alpha(T)$ of $U_p$ are zero modulo $p$ when $p$ is prime, unless the root of $T$ has one child or $p$ children. 
\end{remark}

In Section~\ref{SS:modp:Und} we generalize Theorem \ref{T:modp} to the coefficients $c_{\lambda}^{n,d}$.


\section{Combinatorial interpretations of the $c^n_\lambda$}\label{SS:diagrams}

In this section we investigate combinatorial interpretations of the coefficients $c^n_\lambda$.
A quite useful one follows directly from \eqref{E:trees}: $c^n_{\lambda}$ is the number of trees $T\in \T{n}$  whose internal vertices, other than the root, have as numbers of children the parts of $\lambda$.
 
In what follows, we will investigate another kind of combinatorial interpretation, based on the description of $U_n$ in Section \ref{SS:combint:esdm}. Our starting point is Theorem \ref{T:PD}. Define the \emph{type} of a subdiagonal partial map $g\in\PD{n}$ as the partition whose parts are the cardinalities of its fibers.

\begin{cor}\label{C:subdiagrams}
For any $n$ and any partition $\lambda$, the coefficient $c_{\lambda}^n$ is the number of subdiagonal partial maps from $[n]$ to $[n]$ of type $\lambda$.
\end{cor}
An equivalent description is already given in \cite[Prop.\ 2]{MohammadNoori}, as an ``urns and balls model" (ball $i$ is placed in urn $j$ when $g(i)=j$). The interpretation in terms of partial maps provides an explicit formula for the coefficients $c_{\lambda}^n$. The map
\[
g \in \PD{n} \longmapsto \left(  g^{-1}(\{n-1\}),g^{-1}(\{n-2\}), \ldots, g^{-1}(\{1\}) \right)
\]
establishes a bijection from $\PD{n}$ to the set of all sequences $(R_1, R_2, \ldots,R_{n-1})$ of pairwise disjoint subsets of $[n]$ fulfilling $R_j \subseteq \{n-j+1,n-j+2,\ldots,n\}$, for all $j$. The type of $g$  is the multiset of the nonzero cardinalities of the corresponding sets $R_i$. The sequences $(R_1,R_2,\ldots,R_{n-1})$ with given cardinalities $i_1$, $i_2$, \ldots, $i_{n-1}$ are easily counted:  there are as many as 
\begin{equation}\label{E:diag}
\prod_{j=1}^{n-1} \binom{j - i_1 - i_2 - \ldots - i_{j-1}}{i_j}
\end{equation}
since one can build any such sequence by choosing first $R_{1}$ in $\{n\}$ with $i_1$ elements, then $R_{2} \subseteq \{n-1, n\}$ disjoint from $R_1$ and with $i_2$ elements. Once 
$R_{1}$, $R_{2}$, \ldots, $R_{j-1}$ have been chosen, the $i_j$ elements of $R_j$ have to be selected from the set
\[
\{n-j+1,n-j+2,\ldots, n\} \setminus (R_{1} \sqcup R_{2} \sqcup \cdots \sqcup R_{j-1}),
\]
whose cardinality is $j-i_1-i_2-\ldots-i_{j-1}$. For this, there are $\binom{j-i_1-i_2-\cdots-i_{j-1}}{i_j}$ possibilities.

This yields the following formula.
\begin{cor}\label{C:closed:form:d:is:1}
Let $n\geq k \geq 1$ and $\lambda$ be a partition of $n-k$. Then 
\begin{equation}\label{E:closed:form:d:is:1:a}
c^n_\lambda =\sum_{i_1, \ldots, i_{n-1}}
\prod_{j=1}^{n-1}
{j-i_{1}-\cdots -i_{j-1}\choose i_{j}}.
\end{equation} 
where the sum is carried over all sequences of nonnegative integers whose nonzero terms are the parts of $\lambda$.
\end{cor}
Of course, the sum in the corollary above can be restricted to the sequences $(i_1,i_2,\ldots,i_{n-1})$ such that $i_1+i_2+\ldots+i_{j} \le j$, for all $j$, since the other sequences have a zero contribution. Under this hypothesis, the binomial coefficients ${j-i_{1}-\cdots -i_{j-1}\choose i_{j}}$ in \eqref{E:closed:form:d:is:1:a} can be expanded as $(j-i_{1}-\cdots -i_{j-1})!/i_j!(j-i_1-\cdots-i_j)!$.
Cancellations in the product of binomial coefficients leads to the following formula, due to Comtet \cite[Formula (8)]{Comtet}:
\begin{cor}[\cite{Comtet, BergeronReutenauer, MohammadNoori}]\label{C:Comtet}
Let $n\geq k \geq 1$ and $\lambda$ be a partition of $n-k$. Then 
\begin{equation}\label{E:Comtet}
c^n_\lambda = 
\frac{1}{(k-1)!} \frac{1}{\prod_i (\lambda_i) !} 
\sum_{i_1, \ldots, i_{n-1}} \prod_{j=1}^{n-1} (j-i_1-i_2-\cdots -i_{j-1})
\end{equation}
where the sum if carried over all sequences $(i_1,i_2,\ldots,i_{n-1})$ of nonnegative integers, whose nonzero terms 
 are the parts of $\lambda$, up to reordering, and such that $i_1+i_2+\cdots+i_{j} \le j$, for all $j$. 
\end{cor}


Corollaries \ref{C:closed:form:d:is:1} and \ref{C:Comtet} will be generalized in Corollary \ref{C:closed:form:d}.

\begin{exam}
Consider the case when $\lambda$ has $n-k$ parts, all equal to $1$. We have established in Subsection~\ref{S:combspecial:gs2} that $\displaystyle c^n_{1^{n-k}}$ is the partition Stirling number $\pS{n}{k}$. But $\pS{n}{k}=c^n_{1^{n-k}}$ also counts the subdiagonal partial maps from $[n]$ to $[n]$ of type $1^{n-k}$, which correspond to the \emph{rook placements} of $n-k$ rooks, all below the diagonal, in a $n \times n$ chessboard. This interpretation of the partition Stirling numbers is known, see \cite{Varvak}, and Corollary \ref{C:subdiagrams} is thus a generalization of it.
\end{exam}
\begin{exam}
Let us apply Corollary \ref{C:closed:form:d:is:1} to the partition $\lambda=(1^{n-k})$. Then, the sum in \eqref{E:closed:form:d:is:1:a} is over all bitstrings of length $n-1$ with $n-k$ occurrences of $1$. These sequences encode the $(n-k)$-subsets of $[n-1]$: to $(i_1,i_2,\ldots,i_{n-1})$ corresponds the subset $A=\{a_1, \ldots, a_{n-k} \}$ with $1\leq a_1< \cdots <a_{n-k}\leq n-1$ containing all $a$ such that $i_a=1$.  If $i_j=0$ then $\binom{j -i_1-i_2-\ldots-i_{j-1}}{i_j}=1$ and if $i_j=1$ then $j \in A$, say $j=a_s$, and $i_1+\cdots +i_{j-1}=s-1$. In this case,
\[
{j-i_{1}-\cdots -i_{j-1}\choose i_{j}}=j-i_{1}-\cdots -i_{j-1}=a_s -(s-1).
\]

Thus, for any $n$, $k$ with $1\leq k\leq n$, we get the following formula for the Stirling numbers of the second kind:
\begin{align}\label{E:Stirling:pos}
\pS{n}{k}=\sum_{1\leq a_1< \cdots <a_{n-k}\leq n-1} \prod_{s=1}^{n-k} (a_s -(s-1)) 
\end{align}
By standard algebraic manipulations, \eqref{E:Stirling:pos} can be re-written in the form of \cite[Exer.\ 45]{rS12}.
\end{exam}

\section{Generalization: Normal ordering for $(h \partial^d)^n$}\label{S:gen}

Our approach generalizes naturally to the study of the normal ordering of the operator $(h \partial^d)^n$, for any positive integer $d>0$.

\subsection{Formulas for the normal ordering of $(h \partial^d)^n$}

Recall the derivation $\Delta=\sum_{i\geq 0}y_{i+1}\partial_{y_i}$ of $R\langle t \rangle$. Given any $d\geq 1$, we recursively define a family $U_{n, d}$ of elements of $R\langle t \rangle$, extending the polynomials $U_n$, as follows:
\begin{equation}\label{E:univpolys:d:recdef}
U_{0, d}=1 \quad \text{and} \quad   
\forall n\geq 0,\ U_{n+1, d} = y_0 (\Delta + \rho_t)^d U_{n, d},
\end{equation}
where, as in \eqref{E:univpolys:recdef}, $\rho_t$ stands for the right multiplication by $t$ operator. 

When $d=1$ we retrieve the polynomials $U_n$, and it is not surprising that many of the properties which we have established for the $U_n$ generalize to the $U_{n, d}$. Among these we have that, for $n\geq 1$, $U_{n, d}$ is a sum of monomials of the form $P(y_0, \ldots, y_{(n-1)d})t^k$ with $P\in R$ and $d\leq k\leq nd$ and that $U_{n, d}$ is homogeneous: 
\begin{enumerate}[label=\textup{(\roman*)}]
\item of degree $n$ relative to the grading in which $y_i$ has degree $1$ and $t$ has degree $0$;
\item of degree $nd$ relative to the grading in which $y_i$ has degree $i$ and $t$ has degree $1$.
\end{enumerate}

The main property of these polynomials is the analogue of Theorem~\ref{T:univpolys:univprop}, which we state below.

\begin{thm}\label{T:univpolys:univprop:Und}
For any $n\geq 0$ and $d\geq 1$, there is a unique polynomial $U_{n, d}\in \bigoplus_{i\geq 0}Rt^i\subseteq R\langle t \rangle$ such that, for any ring $A$, derivation $\partial$ of $A$ and central element $h$ in $A$,
\begin{equation}\label{E:univpolys:univprop:Und}
\left(h\partial^d\right)^n 
= \eval{U_{n, d}}{y_i=h^{[i]},\, t=\partial}=U_{n, d}(h, h^{[1]}, h^{[2]}, \ldots; \partial),
\end{equation}
as endomorphisms of $A$.
\end{thm}

\begin{proof}
The proof is identical to the proof of Theorem~\ref{T:univpolys:univprop} and is based on
\begin{equation}\label{E:univpolys:skewrel:d}
\partial^d\circ\ev = \ev \circ (\Delta+ \rho_t)^d
\end{equation}
which follows from \eqref{E:univpolys:skewrel}.
\end{proof}

To avoid repetition, we sketch briefly how the properties of the $U_n$ covered in Sections~\ref{S:univpolys} and~\ref{S:combint} generalize.
\begin{thm}\label{T:Und}
Let $n \ge 0$ and $d\geq 1$. 
\begin{enumerate}[label=\textup{(\alph*)}]
\item The polynomial $U_{n, d}$ is obtained by applying 
to the product
\begin{equation*}
\prod_{i=0}^{n-1} \left( x_{i}+ \cdots + x_2 + x_1 + x_0\right)^d
\end{equation*}
in the commutative polynomial ring $\ZZ[x_0, x_1,\ldots,x_n]$ the $\ZZ$-linear map from $\ZZ[x_0, x_1,\ldots,x_n]$ to $R\langle t \rangle$ defined on a basis of $\ZZ[x_0, x_1,\ldots,x_n]$ by
\[
x_n^{\alpha_n} x_{n-1}^{\alpha_{n-1}} \cdots x_1^{\alpha_1}x_0^k \longmapsto y_{\alpha_n} y_{\alpha_{n-1}} \cdots y_{\alpha_1}t^k.
\]
\item Let $\PD{n,d}$ be the set of all partial maps $[n]\times[d] \rightarrow [n]$ such that $\forall i, j$, $f(i,j) < i$. Then 
\begin{equation}\label{E:Und:Snd}
U_{n, d} = \sum_{f \in \PD{n,d}}   \left(\prod_{i=1}^{n} y_{\#f^{-1}(\{i\})}\right) t^{nd-\#\Dom(f)}. 
\end{equation}
\item For $F=(T_1,\ldots,T_d) \in \T{n}^d$, set $\ch(i;F)=\sum_{j=1}^d\ch(i;T_j)$. Then
\begin{equation}\label{E:Und:Tnd}
U_{n, d} = \sum_{F \in \T{n}^d}  \left( \prod_{i=1}^{n} y_{\ch(i;F)}   \right)  t^{\ch(0;F)}.
\end{equation}
\end{enumerate}
\end{thm}

\begin{proof}
 Firstly, as before, by \eqref{E:umbral:dm} and induction, we obtain
\begin{multline*}
h_n \partial^{d_{n-1}} h_{n-1} \partial^{d_{n-2}} \cdots h_1 \partial^{d_0} (h_0) \\
= 
m_{n+1} \circ
\left(
  \prod_{i=0}^{n-1}   \left(\partial_{i}+\cdots+\partial_{1} + \partial_0\right)^{d_{i}}\right)
  \left(h_{n}  \otimes \cdots \otimes h_1 \otimes h_0\right)
.
\end{multline*}
In particular, for any $a,h \in A$, and any $n, d \ge 0$,
\[
(h \partial^d)^n(a) 
= 
m_{n+1} \circ
\left(
\prod_{i=0}^{n-1} \left(  \partial_{i}+\cdots+\partial_{1} + \partial_0\right)^d \right)
\left(h  \otimes \cdots \otimes h \otimes h \otimes a \right).
\]
Description (a) of $U_{n,d}$ follows from this. Now,
\begin{align*}
\prod_{i=0}^{n-1} \left(  \partial_{i}+\cdots+\partial_{1} + \partial_0\right)^d
&=\left(
\sum_{g \in \PD{n}} \partial_0^{n-\#\Dom(g)} \prod_{i=1}^{n-1} \partial_i^{\#g^{-1}(\{i\})} 
\right)^d\\
&=
\sum_{(g_1,\ldots,g_d) \in (\PD{n})^d} \partial_0^{nd-\sum_i \# \Dom(g_i)} \prod_{i=1}^{n-1}  \partial_i^{\sum_{j} \# g_j^{-1}(\{i\})}.
\end{align*}
This formula becomes \eqref{E:Und:Snd} after making use of the bijection $(\PD{n})^d \cong \PD{n,d} $ that sends $(g_1,\ldots,g_d)$ to the partial map $(i, j) \mapsto g_j(i)$ defined on the pairs $(i,j)$ such that $i \in \Dom(g_j)$. 
Finally, we have also
\begin{align*}
\prod_{i=0}^{n-1} \left(  \partial_{i}+\cdots+\partial_{1} + \partial_0\right)^d
&=\left(
\sum_{T \in \T{n}} \prod_{i=0}^{n-1} \partial_i^{\ch(i;T)}
\right)^d\\
&=
\sum_{(T_1,\ldots,T_d) \in \T{n}^d} \prod_{i=0}^{n-1} \partial_i^{\sum_{j} \ch(i;T_j)}.
\end{align*}
Formula \eqref{E:Und:Tnd} follows from this.
\end{proof}

\subsection{The coefficients $c_{\lambda}^{n,d}$}\label{SS:cnd}

Let us single out the coefficients of $U_{n, d}$, as in the first part of Proposition~\ref{P:recrel}.

\begin{prop}\label{P:recrel:d}
Assume $n, d\geq 1$. There exist positive integers $c^{n, d}_\lambda$, where $\lambda$ runs through the set of partitions of size $0\leq |\lambda|\leq (n-1)d$ with at most $n-1$ parts, such that
\begin{equation}\label{E:und:canonical}
U_{n, d}=\sum_{k=d}^{nd}\sum_{\substack{\lambda\vdash nd-k\\ \ell(\lambda)\leq n-1}} c^{n, d}_\lambda y_0^{n-\ell(\lambda)}y_{\lambda}t^{k}.
\end{equation}
\end{prop}

It is obvious from the definition that $c^{n, 1}_\lambda=c^{n}_\lambda$. By Theorem~\ref{T:Und}, we can give several combinatorial descriptions of the coefficients $c^{n, d}_\lambda$.

Here is a description of $c_{\lambda}^{n,d}$ that generalizes the description of the $c_{\lambda}^n$ as counting subdiagonal partial maps (Theorem \ref{T:PD}). 

\begin{prop}\label{P:PD:d}
The coefficient $c^{n,d}_{\lambda}$ counts the elements of $\PD{n,d}$ of type $\lambda$, i.e., the partial maps $g$ from $[n] \times [d]$ to $[n]$ fulfilling the condition
\[
\forall i, j,\quad g(i,j)<i,
\] 
and whose cardinalities of the fibers form the parts of $\lambda$.
\end{prop}

Below we generalize Corollaries~\ref{C:closed:form:d:is:1} and~\ref{C:Comtet}.
\begin{cor}\label{C:closed:form:d}
Let $n, d\ge 1$, and $k$ such that $nd \ge k \ge d$. Let $\lambda $ be a partition of $nd-k$ with length at most $n-1$. 
\begin{enumerate}[label=\textup{(\alph*)}]
\item We have
\begin{equation}\label{E:gen:d:def:cndl}
c^{n, d}_\lambda=
\sum_{i_1, \ldots, i_{n-1}}
\prod_{j=1}^{n-1}
{jd-i_{1}-\cdots -i_{j-1}\choose i_{j}},
\end{equation}
where the sum is carried over all sequences $i_1, \ldots, i_{n-1}$ of nonnegative integers whose nonzero terms are the parts of $\lambda$ (up to reordering).
\item We have the following generalization of Comtet's formula (Corollary \ref{C:Comtet}):
\begin{equation}\label{E:cnd:Comtet}
c^{n, d}_\lambda=\frac{1}{(k-1)!} \frac{1}{\prod_i (\lambda_i !)}
\sum_{i_1, \ldots, i_{n-1}}
\prod_{j=1}^{n-1}
(jd-i_{1}-\cdots -i_{j-1}),
\end{equation}
where the sum is carried over all sequences $i_1,i_2,\ldots,i_{n-1}$ of nonnegative integers whose nonzero terms are the parts of $\lambda$ (up to reordering) such that $i_1+i_2+\cdots+i_j \le jd$, for all $j$.
\item Let $\M(n)$ be the set of all lower triangular arrays $a=(a_{i,j})_{1 \le j < i \le n}$ of nonnegative integers. To any such array associate the sequence $(r_2(a),r_3(a),\ldots,r_{n}(a))$ of its row sums ($r_i(a)=\sum_j a_{i,j}$), and the sequence $(c_1(a),c_2(a),\ldots, c_{n-1}(a))$ of its column sums ($c_j(a)=\sum_i a_{i,j}$). Then:
\begin{equation}\label{E:inv factorials}
c_{\lambda}^{n,d} = \sum_{a} \frac{1}{\prod_{i,j} (a_{i,j}!)} \prod_{i=2}^n (d)_{r_i(a)},
\end{equation}
where the sum is carried over all $a\in \M(n)$ whose nonzero column sums equal the parts of $\lambda$, up to reordering. 
\end{enumerate}
\end{cor}

\begin{proof}
The proofs of (a) and (b) are straightforward adaptations of those of Corollaries \ref{C:closed:form:d:is:1} and \ref{C:Comtet}, and as such are omitted.

Let us prove (c). Associate to any $g\in\PD{n,d}$ the array $a=(a_{i,j})\in \M(n)$ such that $a_{i,j}=\# \{k \in [d]\,|\,g(i,k)=j\}$. Given $a \in \M(n)$, the elements $g \in \PD{n,d}$ corresponding to $a$ are obtained by: 
\begin{itemize}
\item Firstly, choosing for each $i$, a subset $S_i \subseteq [d]$ with $r_i(a)$ elements (to be the set of all $k$ such that $(i,k) \in\Dom(g)$). There are $\binom{d}{r_i(a)}$ possibilities for this choice.
\item Next, choosing for each $i$, and each $k \in S_i$, the value of $g(i,k)$ in $[i-1]$,
so that each $j \in [i-1]$ is chosen $a_{i,j}$ times. There are $\binom{r_i(a)}{a_{i,1},a_{i,2},\ldots,a_{i,i-1}}$ possibilities for this choice. 
\end{itemize}
This yields the formula
\[
c_{\lambda}^{n,d} = \sum_a \prod_i \binom{d}{r_i(a)} \binom{r_i(a)}{a_{i,1},a_{i,2},\ldots, a_{i,i-1}},
\]
where the sum is carried over all $a\in \M(n)$ with nonzero column sums equal to the parts of $\lambda$. This formula simplifies into that of the corollary.
\end{proof}

\begin{exam}\label{Ex:c33}
Let us compute $c^{3, 3}_\lambda$, i.e., the coefficients occurring in the normally ordered form of $(h \partial^3)^3$, using \eqref{E:gen:d:def:cndl}. In general, we need to determine all different $(n-1)$-tuples $(i_1, \ldots, i_{n-1})$ which can be obtained by permuting the parts of $\lambda$ and adding zero parts, if necessary. However, these terms will give a zero contribution unless 
\begin{equation}\label{E:gen:d:admissible:comp}
i_1+\cdots+i_j\leq jd, \quad \mbox{for all $1\leq j\leq n-1$,} 
\end{equation}
which narrows down the number of $(n-1)$-tuples to consider, as will now be illustrated.

We need to consider partitions of size at most $6$ with no more than $2$ parts.
For example, if $\lambda=4$, then the possibilities are $(4, 0)$ and $(0, 4)$, but the former violates \eqref{E:gen:d:admissible:comp}, so we get
$$
c^{3, 3}_4={3\choose 0}{6\choose 4}=15.
$$
Similarly, for $\lambda=(2, 1)$ the possibilities are $(2, 1)$ and $(1, 2)$, giving
$$
c^{3, 3}_{2, 1}={3\choose 2}{4\choose 1}+{3\choose 1}{5\choose 2}=42.
$$
Proceeding as above, we obtain all of the coefficients $c^{3, 3}_\lambda$ shown below.
\[
\begin{array}{c|*{16}{c}} 
\lambda & {\emptyset} & { \Yboxdim{4pt} \yng(1)} & { \Yboxdim{4pt} \yng(1,1)} & { \Yboxdim{4pt} \yng(2)} & { \Yboxdim{4pt} \yng(2,1)} & { \Yboxdim{4pt} \yng(3)} & { \Yboxdim{4pt} \yng(3,1)} & { \Yboxdim{4pt} \yng(2,2)}& { \Yboxdim{4pt} \yng(4)}& { \Yboxdim{4pt} \yng(4,1)}& { \Yboxdim{4pt} \yng(3,2)}& { \Yboxdim{4pt} \yng(5)}& { \Yboxdim{4pt} \yng(6)}& { \Yboxdim{4pt} \yng(5,1)}& {\Yboxdim{4pt} \yng(4,2)}& { \Yboxdim{4pt} \yng(3,3)}\\[4pt]
 \hline
 \vphantom{\displaystyle \int} c^{3, 3}_{\lambda}&1 &9 &15& 18 &42& 21& 33 &18&15&15&15&6&1&3&3&1
\end{array}
\]
Note that $\displaystyle \sum_{\lambda} c^{3, 3}_{\lambda}=216=(3!)^3$.
\end{exam}

\begin{cor}\label{C:sumc:d:is:factorial}
Let $n, d\geq 1$. Then
\begin{equation*}
\sum_{k=d}^{nd}\sum_{\lambda\vdash nd-k} c^{n, d}_\lambda =n!^d \quad\mbox{and}\quad
\sum_{\lambda\vdash (n-1)d} c^{n, d}_\lambda =(n-1)!^d. 
\end{equation*}
\end{cor}
\begin{proof}
We will use the description of $c^{n, d}_\lambda$ provided in Proposition \ref{P:PD:d}. The first sum is just $\#\PD{n, d}=n!^d$. For the second sum, observe that if $g\in\PD{n, d}$ has type $\lambda$ then $\#\Dom(g)=|\lambda|$ and $(n-1)d$ is the maxim possible cardinality of the domain of an element in $\PD{n, d}$. So, the second sum counts the number of functions $g$ from $([n]\setminus\{1\}) \times [d]$ to $[n]$ satisfying $g(i,k)<i$ for all $i, k$, which equals 
$(n-1)!^d$.
\end{proof}

\subsection{Application to generalized Stirling numbers}\label{SS:g S}

To finish this section, we extend the result of Section \ref{S:combspecial:gs2} giving the generalized Stirling numbers $\pS{n}{k}_{q,1}$ as the coefficients of $U_n((q)_0,(q)_1,(q)_2,\ldots;t)$.
\begin{prop}\label{P:stirling q d}
Let $q$, $d$, $n$, $k$ be nonnegative integers, with $q \ge d$. The generalized Stirling number $\pS{n}{k}_{q,d}$, defined by \eqref{E:def:Stirling}, is given by
\begin{equation}\label{E:stirling q d}
\pS{n}{k}_{q,d} = \sum_{\substack{\lambda\vdash nd-k\\ \ell(\lambda)\leq n-1}} c_{\lambda}^{n,d} \prod_i (q)_{\lambda_i},
\end{equation}
where $(q)_j$ stands for the falling factorial $q(q-1)\cdots (q-j+1)$. 
\end{prop}

\begin{proof}
Take $A=\ZZ[x]$ and $h=x^q$ in \eqref{E:univpolys:univprop:Und}, and compare with \eqref{E:def:Stirling}. This yields 
\[
U_{n,d}\left((x^q),(x^q)',(x^q)'',\ldots; \partial_x\right) =  x^{(q-d) n} \sum_k \pS{n}{k}_{q,d} x^k \partial_x^k.
\]
Setting $x=1$, this gives
\[
U_{n,d}\left((q)_0,(q)_1,(q)_2,\ldots; \partial_x\right) =  \sum_k \pS{n}{k}_{q,d} \partial_x^k,
\]
from which the result follows.
\end{proof}
The Stirling numbers $\pS{n}{k}_{q,d}$ are symmetric in $q$ and $d$, a fact that is not reflected by Formula \eqref{E:stirling q d}. Below we transform Formula \eqref{E:stirling q d} to make this symmetry explicit. The nice formula obtained this way seems new.
\begin{cor}
Let $q$, $d$, $n$, $k$ be nonnegative integers, with $q \ge d$. The generalized Stirling number $\pS{n}{k}_{q,d}$ is given by
\begin{equation}
\pS{n}{k}_{q,d} = \sum_{a\in\M(n)}  \frac{\prod_i (d)_{r_i(a)} \cdot \prod_j (q)_{c_j(a)}}{\prod_{i,j} (a_{i,j} !)} ,
\end{equation}
where $\M(n)$ is the set of all lower triangular arrays $a=(a_{i,j})_{1 \le j < i \le n}$ of nonnegative integers with $c_j(a)=\sum_i a_{i,j}$ (column sum) and $r_i(a)=\sum_j a_{i,j}$ (row sum).
\end{cor}
\begin{proof}
This is the immediate consequence of Proposition \ref{P:stirling q d} and Formula \eqref{E:inv factorials}.
\end{proof}

It is natural to ask for a combinatorial interpretation of the right-hand side of Formula \eqref{E:stirling q d}. We will see that we recover A.\ Varvak's description (\cite[Cor.\ 3.2]{Varvak}) of the generalized Stirling numbers $\pS{n}{k}_{q,d}$. 
We recall that a \emph{partial bijection} from a set $A$ to a set $B$ is a bijection $g$ from some subset $\Dom(g)$ of $A$ to some subset of $B$.

\begin{cor}\label{C:comb:gSqd}
Let $n$, $k$, $d$, $q$ be nonnegative integers, with $q \ge d$. The generalized Stirling number $\pS{n}{k}_{q,d}$ counts the partial bijections  $g$ from $[n] \times [d]$ to $[n] \times [q]$ with the property
\[
\forall (i,a,j,b) \in [n]\times[d]\times[n]\times[q], \quad g(i,a)=(j,b) \Rightarrow j < i,
\]
and such that $\#\Dom(g)=nd-k$.
\end{cor}
\begin{proof}
The number $c_{\lambda}^{n,d}$ counts the partial maps $g: [n] \times [d] \rightarrow [n]$ that are in $\PD{n,d}$, and  whose fibers have cardinalities equal to the parts of $\lambda$. For each such $g$, the product $\prod_i (q)_{\lambda_i}$ counts the maps $\phi$ from $\Dom(g)$ to $[q]$ that are injective on each fiber of $g$. They are exactly the maps $\phi: \Dom(g) \rightarrow [q]$ such that the map $(i,j)\in \Dom(g) \mapsto (g(i,j),\phi(i,j)) \in [n] \times [q]$ is injective.  
The corollary follows from this.
\end{proof}

The bijections 
\[
\begin{array}{rclcrcl}
[n] \times [d] &\longrightarrow & [nd]  &\textrm{\ and\ }   & [n] \times [q] &\longrightarrow & [nq] \\
(i,a)          & \longmapsto& (i-1)d+a && (j, b)                & \longmapsto& (j-1)q+b
\end{array}
\]
turn the graphs of the partial bijections from Corollary~\ref{C:comb:gSqd} into the rook placements of $nd-k$ rooks in a chessboard whose shape is the Young diagram of the partition whose parts are $q$, $2 q$, \ldots, $(n-1)q$, $nq$, each with multiplicity $d$.

Therefore, the Stirling number $\pS{n}{k}_{q,d}$ counts these rook placements. This description of $\pS{n}{k}_{q,d}$ is the one given in \cite[Cor.\ 3.2]{Varvak}.

\subsection{The coefficients $c_{\lambda}^{n,d}$ modulo a prime $p$, for $d=p^m$}\label{SS:modp:Und}

Theorem \ref{T:modp} admits the following generalization.

\begin{thm}\label{T:modp:Und}
Let $p$ be a prime and assume that $d$ is a power of $p$. Given a partition $\lambda$, let $d{\mkern -1mu\cdot\mkern -1mu}\lambda$ be the partition obtained by multiplying every part of $\lambda$ by $d$. Then:
\begin{enumerate}[label=\textup{(\alph*)}]
\item $c^{n, d}_{d{\mkern -1mu\cdot\mkern -1mu}\lambda} \equiv c^{n}_\lambda\ (\modd p)$;
\item $c^{n, d}_\mu \equiv 0\ (\modd p)$ if $\mu\neq d{\mkern -1mu\cdot\mkern -1mu}\lambda$, for all $\lambda$.
\end{enumerate}
In particular, if both $n$ and $d$ are powers of $p$, then $c^{n, d}_\mu \equiv 0\ (\modd p)$, as long as $|\mu|\neq d(n-1)$ and $|\mu|$ is not a multiple of $dp$.
\end{thm}

\begin{proof}
The cyclic group $\ZZ_d$ acts on on the set $\T{n}^d$ of $d$-tuples of increasing trees $(T_1, \ldots, T_d)$ by cyclic place permutation. The singular orbits are those of the form $(T, T, \ldots, T)$, for $T\in\T{n}$. If $d$ is a power of a prime $p$, then all other orbits have size divisible by $p$. Thus, 
\begin{equation}\label{E:Und:Tnd:dpp}
U_{n, d} \equiv \sum_{T \in \T{n}}  \left( \prod_{i=1}^{n} y_{d\ch(i;T)}   \right)  t^{d\ch(0;T)}\ \ (\modd p).
\end{equation}
We deduce (a) and (b) by equating coefficients in \eqref{E:Und:Tnd:dpp}.
The last statement of the theorem follows from the above and Theorem \ref{T:modp}.
\end{proof}

\begin{remark}
Observe that Theorem~\ref{T:modp:Und} above correctly computes the congruence classes modulo $3$ of all the coefficients $c^{3, 3}_\lambda$, knowing only that $c^3_\emptyset=c^3_{(2)}=c^3_{(1^2)}=1$ (see Example~\ref{Ex:c33}).
\end{remark}


\section{Normal ordering in formal differential operator rings}\label{S:nord:fdor}

It is well known that, over the polynomial ring $\FF[x]$, if the base field $\FF$ has characteristic $0$, then the operators $\partial_x$ and $x$ (multiplication by $x$) generate the Weyl algebra $\A_1=\FF\langle x, y\mid [y, x]=1\rangle$. Hence, the normally ordered form
of the operator $(h \partial^d)^n$, which we have been discussing, yields in particular known expressions for the normally ordered form of elements of this type in the Weyl algebra, with $\partial=\partial_x$ and $h\in\FF[x]$ (see \cite{Blasiak:thesis}, \cite{BlasiakFlajolet} and \cite{MansourSchork}). In this section we will apply our results to the more general setting of formal differential operator rings, which include in particular the subalgebras $\A_h$ of the Weyl algebra studied in \cite{blo13, blo15ja, blo15tams} and defined below in \eqref{E:def:Ah}. 

Recall that a \emph{formal differential operator ring} (or \emph{skew polynomial ring}) is a ring, denoted $A[z; \partial]$, constructed from a ring $A$ and a derivation $\partial$ of $A$, whose elements can be expressed as polynomials $\sum_{i=0}^n a_i z^i$ in a new variable $z\in A[z; \partial]$, with $n\geq 0$ and uniquely determined coefficients $a_i\in A$. Thus, as a left $A$-module, $A[z; \partial]=\bigoplus_{i\geq 0}Az^i$ is free with basis $\{ z^i\}_{i\geq 0}$ and can be identified as such with the polynomial ring $A[z]$. However, the variable $z$ does not necessarily commute with the coefficients from $A$ and multiplication in $A[z, \partial]$ is determined by the multiplication in $A$, associativity, the distributive law and the commutation relation $za=az+\partial(a)$, for all $a\in A$. In particular, the adjoint map $[z, -]:A\longrightarrow A$ given by the commutator $[z,a]=za-az$ for $a\in A$ coincides with the derivation $\partial$. See \cite{GW89} or \cite{MR00} for more details.

An important example of a formal differential operator ring is the first Weyl algebra $\A_1=\FF[x]\left[y;\partial_x \right]$ over the field $\FF$ where, in case $\chara(\FF)=0$, $x$ can be identified with left multiplication by $x$ on $\FF[x]$ and $y$ with $\partial_x$. In case $\chara(\FF)=p>0$, then this correspondence does not give a faithful representation of $\A_1$ on $\FF[x]$ because, as an operator, $\partial_x^p=0$, whereas in $\A_1$ we have $y^p\neq 0$. Another example is given by the family of algebras $\A_h=\FF[x][\hat y;\delta]$, studied in \cite{blo13, blo15ja, blo15tams} and determined by the derivation $\delta=h\partial_x$, where $h\in\FF[x]$. Then $\A_h$ admits the presentation 
\begin{equation}\label{E:def:Ah}
\A_h=\FF\langle x, \hat y\mid [\hat y, x]=h(x)\rangle. 
\end{equation}
By taking $h=1$ we retrieve the Weyl algebra $\A_1$ and, for an arbitrary nonzero $h\in\FF[x]$, we can view $\A_h$ as the (unital) $\FF$-subalgebra of $\A_1$ generated by $x$ and $yh$, under the identification $\hat y=yh$. It is clear that, as $h$ ranges over $\FF[x]$, the family of algebras $\A_h$ runs through all formal differential operator rings over $\FF[x]$.

As before, let $R=\ZZ[y_i; i\geq 0]$ and $\Delta=\sum_{i\geq 0}y_{i+1}\partial_{y_i}$ be the derivation of $R$ satisfying $\Delta(y_i)=y_{i+1}$ for all $i\geq 0$. The differential operator ring $R[z;\Delta]$ can also be seen as the ring generated by the commuting variables $y_i$, for $i \in \ZZ$, and a new (non-central) variable $z$ satisfying the commutation relations $z y_i=y_i z+y_{i+1}$, for all $i\geq 0$. In fact, it is natural to view $R[z;\Delta]$ as a homomorphic image of $R\langle t \rangle$ under the epimorphism $R\langle t \rangle\twoheadrightarrow R[z;\Delta]$ sending $y_i$ to $y_i$ and $t$ to $z$. Under this map we can, and will, think of the polynomials $U_{n, d}$ as elements of $R[z;\Delta]$, substituting $z$ for $t$. Fixing a base field $\FF$ (of arbitrary characteristic), we can also consider the corresponding $\FF$-algebras $R_\FF:=\FF\otimes_\ZZ R=\FF[y_i; i\geq 0]$ and $\FF\otimes_\ZZ R[z;\Delta]=R_\FF[z; \Delta]$.

The main result of this section, presented below, shows that the polynomials $U_{n, d}$ give the normally ordered form of certain elements over a formal differential operator ring.

\begin{thm}\label{T:gen:d}
Let $A$ be a ring, $h$ a central element of $A$ and $\partial$ a derivation of $A$. Then, for all $n, d\geq 1$ we have the following normal ordering identity in the formal differential operator ring $A[z; \partial]$:
\begin{align}\label{E:gen:d:comp}
(h z^d)^n= \eval{U_{n, d}}{y_i=h^{[i]},\, t=z}
=\sum_{k=d}^{nd}\sum_{\lambda\vdash nd-k} c^{n, d}_\lambda h^{n-\ell(\lambda)}h^{[\lambda]}z^{k},
\end{align}
where $h^{[\lambda]}=h^{[\lambda_1]}\cdots h^{[\lambda_\ell]}$, for $\lambda=(\lambda_1, \ldots, \lambda_\ell)$.
\end{thm}

\begin{proof}
We will first prove the special case of the Theorem for the ring $R$, the derivation $\Delta$ and the central element $h=y_0\in R$ and then deduce the general case.  

There is a representation $\rho: R[z; \Delta]\longrightarrow \operatorname{End}_\ZZ(R)$ sending $r\in R$ to the corresponding left multiplication by $r$ endomorphism, which we still denote by $r$, and sending $z$ to $\Delta$. This is a ring endomorphism precisely because $\Delta$ is a derivation of $R$. 

Then, by Theorem~\ref{T:univpolys:univprop:Und}, we have the identity 
\begin{equation*}
\rho\left( \left(y_0z^d\right)^n\right)=\left(y_0\Delta^d\right)^n = \eval{U_{n, d}}{t=\Delta}  = \rho\left( \eval{U_{n, d}}{t=z} \right)
\end{equation*}
in $\operatorname{End}_\ZZ(R)$, as $y_0^{[i]}=y_i$. Hence, it remains to show that $\rho$ is injective. Suppose $\rho\left(\sum_{i=0}^{m}P_{i} z^i\right)=0$. Then $\sum_{i=0}^{m}P_{i} \Delta^i=0$ as an endomorphism of $R$. By the uniqueness argument from the proof of Theorem~\ref{T:univpolys:univprop}, we conclude that $P_i=0$ for all $i$ and so $\sum_{i=0}^{m}P_{i} z^i=0$, thus establishing the injectivity of $\rho$.

Now consider the general case of the Theorem, with $A$, $h$ and $\partial$ as in the statement. There is a ring homomorphism $\phi:R[z; \Delta]\longrightarrow A[z; \partial]$ defined on the generators by $y_i\mapsto h^{[i]}$ and $z\mapsto z$. This is well defined precisely because $h^{[i]}$ and $h^{[j]}$ commute, for all $i, j\geq 0$, and by the multiplication rule in a formal differential operator ring. Applying $\phi$ to the identity $\left(y_0z^d\right)^n=\eval{U_{n, d}}{t=z}$ yields precisely $\left(hz^d\right)^n=\eval{U_{n, d}}{y_i=h^{[i]},\, t=z}$. The right-hand side of \eqref{E:gen:d:comp} follows directly from \eqref{E:und:canonical}.
\end{proof}

Fix a field $\FF$. For $0\neq h\in\FF[x]$, view the algebra $\A_h$ defined in \ref{E:def:Ah} as a subalgebra of the Weyl algebra $\A_1$ by identifying $\hat y\in\A_h$ with $yh\in\A_1$, as explained above. 

Then $\A_h$ is a free (left) $\FF[x]$-module and each of the following is a free basis (compare~\cite{blo15tams}): 
\begin{equation}\label{E:Ah:bases}
\left( \hat y^n\right)_{n\geq 0},\qquad \left( (hy)^n\right)_{n\geq 0}\qquad \mbox{and}\quad \left( h^n y^n\right)_{n\geq 0}. 
\end{equation}
In fact, fixing $n\geq 0$, each of $\left( \hat y^k\right)_{0\leq k\leq n}$, $\left( (hy)^k\right)_{0\leq k\leq n}$, and $\left( h^k y^k\right)_{0\leq k\leq n}$ spans the same (left) $\FF[x]$-submodule of $\A_h$. 

The Weyl algebra identity 
\begin{equation}\label{E:basis:trans:hyn:hkyk}
(hy)^n=\sum_{k=1}^n\sum_{\lambda\vdash n-k} c^n_\lambda h^{n-k-\ell(\lambda)}h^{(\lambda)}h^{k}y^{k},
\end{equation}
which follows by direct application of Theorem~\ref{T:gen:d}, relates the second and third bases in \eqref{E:Ah:bases}. Furthermore, the relation $(hy)^{n+1}=h(yh)^ny=h(\hat y^n) y$ and the fact that $\A_1$ is a domain yield
\begin{equation}\label{E:basis:trans:hatyn:hkyk}
\hat y^n=\sum_{k=0}^n\sum_{\lambda\vdash n-k} c^{n+1}_\lambda h^{n-k-\ell(\lambda)}h^{(\lambda)}h^{k}y^k ,
\end{equation}
relating the first and the third bases in \eqref{E:Ah:bases}.

\begin{exam}
Take $h=x$. The algebra $\A_x$ is the universal enveloping algebra of the non-abelian $2$-dimensional Lie algebra. We have, by the binomial theorem:
\begin{equation*}
\hat y^n=(yx)^n= (xy+1)^n=\sum_{k=0}^n {n\choose k}(xy)^k.
\end{equation*}
On the other hand, we can also apply the formula \eqref{E:basis:trans:hatyn:hkyk}, taking into account that $h^{(\lambda)}=0$ unless $\lambda=1^\ell$ is a partition with no parts greater than $1$. Thus,
\begin{equation}\label{Exam:Ax:haty}
\hat y^n=\sum_{k=0}^n c^{n+1}_{1^{n-k}} x^{k}y^k=\sum_{k=0}^n \left\{ {n+1\atop k+1}\right\} x^{k}y^k,
\end{equation}
yielding the identity
\begin{equation*}
\sum_{k=0}^n {n\choose k}(xy)^k= \sum_{k=0}^n \left\{ {n+1\atop k+1}\right\} x^{k}y^k.
\end{equation*}
Now, using \eqref{E:basis:trans:hyn:hkyk} we get, for $k\geq 1$,
\begin{equation*}
(xy)^k=\sum_{j=1}^k c^k_{1^{k-j}} x^{j}y^{j} =\sum_{j=1}^k \left\{ {k\atop j}\right\} x^{j}y^{j}.
\end{equation*}
Therefore, we deduce the identity
\begin{align*}
\hat y^n &= 1+\sum_{k=1}^n {n\choose k}(xy)^k=1+\sum_{k=1}^n \sum_{j=1}^k {n\choose k}\left\{ {k\atop j}\right\} x^{j}y^{j}\\&=1+\sum_{j=1}^n \left( \sum_{k=j}^n  {n\choose k}\left\{ {k\atop j}\right\} \right)x^{j}y^{j}.
\end{align*}
Comparing with \eqref{Exam:Ax:haty} we obtain the identity 
\begin{equation}\label{Exam:Ax:binomstirling}
\left\{ {n+1\atop k+1}\right\}= \sum_{j=k}^n  {n\choose j}\left\{ {j\atop k}\right\},
\end{equation}
which appears in \cite[Sec.\ 1.2.6 (52)]{dK97}. 

In terms of the algebra $\A_x$, equation \eqref{Exam:Ax:binomstirling} is a consequence of the change of basis relations between the $\FF[x]$-bases of $\A_x$ mentioned in \eqref{E:Ah:bases}. Indeed, the matrices $A_n=\left( a_{i, j} \right)_{0\leq i, j\leq n}$, $B_n=\left( b_{i, j} \right)_{0\leq i, j\leq n}$ and $C_n=\left( c_{i, j} \right)_{0\leq i, j\leq n}$, with $a_{i, j}=\left({j\atop i}\right)$, $b_{i, j}=\left\{ {j\atop i}\right\}$ and $c_{i, j}=\left\{ {j+1\atop i+1}\right\}$, are the transition matrices from $\left( \hat y^k\right)_{0\leq k\leq n}$ to $\left( (xy)^k\right)_{0\leq k\leq n}$, from $\left( (xy)^k\right)_{0\leq k\leq n}$ to $\left( x^k y^k\right)_{0\leq k\leq n}$ and from $\left( \hat y^k\right)_{0\leq k\leq n}$ to $\left( x^k y^k\right)_{0\leq k\leq n}$, respectively. Thus, \eqref{Exam:Ax:binomstirling} is just the statement that
\begin{equation*}
 \left( \left\{ {j\atop i}\right\} \right)_{0\leq i, j\leq n} \cdot \left( \left({j\atop i}\right) \right)_{0\leq i, j\leq n}=
 \left( \left\{ {j+1\atop i+1}\right\} \right)_{0\leq i, j\leq n},\qquad \mbox{for all $n\geq 0$.}
\end{equation*}


\end{exam}


\section{Final remarks}\label{S:fr}

\paragraph{Effective calculation of the formulas.}
Flajolet and Blasiak remarked in \cite[appx.\ A]{BlasiakFlajolet} that ``the symbolic problem [of calculating $U_n$] is indeed difficult''. Still, they provide a recipe for computer calculation, but it involves several steps, among them taking the compositional inverse of a power series by Lagrange inversion and extracting coefficients from powers of power series. The various descriptions of the $U_n$ given in Section \ref{S:combint} show that this calculation can be done in simpler ways. 

\paragraph{Noncommutative setting.} It is possible to drop the hypothesis ``$h$ central''. We obtain that there is a unique polynomial $V_n(y_{0},y_{1},\ldots;t)$ in non-commuting variables $y_i$, such that for any ring $A$, derivation $\partial$ of $A$  and element $h \in A$, $(h \partial)^n=V_n(h,h^{[1]},\ldots;\partial)$. For instance:
\begin{align*}
V_0 &= 1, \qquad V_1 = y_0 t, \qquad V_2 = y_0 y_1 t + y_0^2 t^2, \\
V_3 &= y_0 y_1^2 t + y_0^2 y_2 t + 2 y_0^2 y_1 t^2 + y_0 y_1 y_0 t^2+ y_0^3 t^3.
\end{align*}
Then many of the descriptions of the $U_n$ are still valid for the $V_n$ with simple changes:
\begin{itemize}
\item Inductive formula: $V_0=1, V_{n+1}= y_0(\Delta + \rho_t) V_n$, where now $\Delta=\sum_{i=1}^{\infty} D_i$, and $D_i$ is the operator on words in $y_0, y_1, y_2, \ldots$ that replaces the letter $y_j$ in position $i$ by the next letter $y_{j+1}$.
\item Umbral formula: $V_n$ is obtained from $\prod_{i=0}^{n-1} (x_i+\cdots+x_1+x_0)$ by applying the linear map sending $x_n^{\alpha_n} x_{n-1}^{\alpha_{n-1}}\cdots x_1^{\alpha_1} x_0^k$ to $y_{\alpha_n} y_{\alpha_{n-1}} \cdots y_{\alpha_1} t^k$.
\item Formula from increasing trees: 
\[
V_n = \sum_{T \in \T{n}} y_{\ch(n,T)} y_{\ch(n-1,T)} \cdots y_{\ch(1,T)} t^{ch(0,T)}.
\]
\end{itemize}
The polynomial $V_n$ has an expansion of the form 
\[
V_n = \sum c^{n}_\alpha y_0 y_{\alpha_{n-1}} \cdots y_{\alpha_1} t^{n-\sum_i \alpha_i}
\]
where the sum is carried over all sequences $\alpha = (\alpha_1,\alpha_2,\ldots,\alpha_{n-1})$ of nonnegative integers with sum at most $n$. The coefficient $c_{\alpha}^n$ interprets as the number of subdiagonal partial maps $[n] \rightarrow [n]$ with $\alpha_1$ elements mapped to $1$, $\alpha_2$ elements mapped to $2$, \ldots, $\alpha_{n-1}$ mapped to $n-1$.

The same straightforward extensions hold for the polynomials $U_{n,d}$ as well.

\paragraph{Similarities with the Fa\`a Di Bruno Formulas.}
The polynomials $U_n$ are defined by the simple formula: 
$U_n = (y_0(\Delta+\rho_t))^n(1)$, with $ \Delta=\sum_i y_{i+1} \frac{d \phantom{f}}{d yi}$.
The \emph{Fa\`a Di Bruno formulas} giving the derivatives $z^{(n)}$ of the composition 
$z=x \circ y$ of two functions, in terms of $x$ and the derivatives of $y$,
can be described similarly: for each $n$, there is a polynomial $F_n(y_1,y_2,y_3,\ldots; t)$ such that 
\[
z^{(n)} =  F_n(y',y'',y''',\ldots; \partial_u)(x) \circ y.
\]
The polynomials $F_n$ are defined by $F_n = (\Delta+ t y_1)^n(1)$.

\paragraph{Normal ordering of powers of particular differential operators.} The normal ordering of $(h \partial)^n$ or $(h \partial^d)^n$ for \emph{particular} functions $h$ (by contrast with the calculation of the \emph{symbolic} formulas $U_n$ where $h, h^{[1]}, h^{[2]} \ldots$ have to be plugged in) is a close, but different problem, for which we refer to the broad combinatorial framework designed by Blasiak and Flajolet \cite{BlasiakFlajolet}.




\begin{thebibliography}{10}

\bibitem{blo13}
Georgia Benkart, Samuel~A. Lopes, and Matthew Ondrus.
\newblock A parametric family of subalgebras of the {W}eyl algebra {II}.
  {I}rreducible modules.
\newblock In {\em Recent developments in algebraic and combinatorial aspects of
  representation theory}, volume 602 of {\em Contemp. Math.}, pages 73--98.
  Amer. Math. Soc., Providence, RI, 2013.

\bibitem{blo15ja}
Georgia Benkart, Samuel~A. Lopes, and Matthew Ondrus.
\newblock Derivations of a parametric family of subalgebras of the {W}eyl
  algebra.
\newblock {\em J. Algebra}, 424:46--97, 2015.

\bibitem{blo15tams}
Georgia Benkart, Samuel~A. Lopes, and Matthew Ondrus.
\newblock A parametric family of subalgebras of the {W}eyl algebra {I}.
  {S}tructure and automorphisms.
\newblock {\em Trans. Amer. Math. Soc.}, 367(3):1993--2021, 2015.

\bibitem{BergeronReutenauer}
Fran\c{c}ois Bergeron and Christophe Reutenauer.
\newblock Une interpr\'etation combinatoire des puissances d'un op\'erateur
  diff\'erentiel lin\'eaire.
\newblock {\em Ann. Sci. Math. Qu\'ebec}, 11(2):269--278, 1987.

\bibitem{Blasiak:thesis}
Pawe{{\l}} B{{\l}}asiak.
\newblock {\em Combinatorics of boson normal ordering and some applications}.
\newblock PhD thesis, Institute of Nuclear Physics of Polish Academy of
  Science, Krakow, and Universit\'{e} Pierre et Marie Curie, Paris, 2005.

\bibitem{BlasiakFlajolet}
Pawe{{\l}} B{{\l}}asiak and Philippe Flajolet.
\newblock Combinatorial models of creation-annihilation.
\newblock {\em S\'em. Lothar. Combin.}, 65:Art. B65c, 78, 2010/12.

\bibitem{BroadhurstKreimer}
D.~J. Broadhurst and D.~Kreimer.
\newblock Renormalization automated by {H}opf algebra.
\newblock {\em J. Symbolic Comput.}, 27(6):581--600, 1999.

\bibitem{Cayley}
Arthur Cayley.
\newblock On the theory of the analytical forms called trees.
\newblock {\em Philosophical Magazine}, 4(13):172--176, 1857.

\bibitem{Comtet}
Louis Comtet.
\newblock Une formule explicite pour les puissances successives de
  l'op\'erateur de d\'erivation de {L}ie.
\newblock {\em C. R. Acad. Sci. Paris S\'er. A-B}, 276:A165--A168, janvier
  1973.

\bibitem{EngbersGalvinHilyard}
John Engbers, David Galvin, and Justin Hilyard.
\newblock Combinatorially interpreting generalized {S}tirling numbers.
\newblock {\em European Journal of Combinatorics}, 43:32--54, 2015.

\bibitem{GW89}
K.~R. Goodearl and R.~B. Warfield, Jr.
\newblock {\em An introduction to noncommutative {N}oetherian rings}, volume~16
  of {\em London Mathematical Society Student Texts}.
\newblock Cambridge University Press, Cambridge, 1989.

\bibitem{HNT}
Florent Hivert, Jean-Christophe Novelli, and Jean-Yves Thibon.
\newblock Multivariate generalizations of the {F}oata-{S}ch\"utzenberger
  equidistribution.
\newblock In {\em Fourth {C}olloquium on {M}athematics and {C}omputer {S}cience
  {A}lgorithms, {T}rees, {C}ombinatorics and {P}robabilities}, volume~AG of
  {\em DMTCS Proceedings}, pages 289--299, Nancy, 2006. Assoc. Discrete Math.
  Theor. Comput. Sci.

\bibitem{oeis}
OEIS~Foundation Inc.
\newblock The {O}n-{L}ine {E}ncyclopedia of {I}nteger {S}equences, 2018.

\bibitem{dK97}
Donald~E. Knuth.
\newblock {\em The art of computer programming. {V}ol. 1}.
\newblock Addison-Wesley, Reading, MA, 1997.
\newblock Fundamental algorithms, Third edition.

\bibitem{MansourSchork}
Toufik Mansour and Matthias Schork.
\newblock {\em Commutation relations, normal ordering, and {S}tirling numbers}.
\newblock Discrete Mathematics and its Applications (Boca Raton). CRC Press,
  Boca Raton, FL, 2016.

\bibitem{MR00}
J.~C. McConnell and J.~C. Robson.
\newblock {\em Noncommutative {N}oetherian rings}, volume~30 of {\em Graduate
  Studies in Mathematics}.
\newblock American Mathematical Society, Providence, RI, revised edition, 2001.
\newblock With the cooperation of L. W. Small.

\bibitem{MohammadNoori}
M.~Mohammad-Noori.
\newblock Some remarks about the derivation operator and generalized {S}tirling
  numbers.
\newblock {\em Ars Combin.}, 100:177--192, 2011.


\bibitem{Scherk}
Heinrich~F.\ Scherk.
\newblock {\em De evolvenda functione $(yd.yd.yd...ydX/dx^n)$ disquisitiones
  nonnullae analyticae}.
\newblock PhD thesis, Berlin, 1823.
\newblock Available online from G{\"o}ttinger Digitalisierungszentrum (GDZ).


\bibitem{rS12}
Richard~P. Stanley.
\newblock {\em Enumerative combinatorics. {V}olume 1}, volume~49 of {\em
  Cambridge Studies in Advanced Mathematics}.
\newblock Cambridge University Press, Cambridge, second edition, 2012.

\bibitem{Varvak}
Anna Varvak.
\newblock Rook numbers and the normal ordering problem.
\newblock {\em J. Combin. Theory Ser. A}, 112(2):292--307, 2005.

\end{thebibliography}
\end{document}